\documentclass{article}
\usepackage{float}
\usepackage[T1]{fontenc}
\usepackage[utf8]{inputenc}
\usepackage{amsmath}
\usepackage{amsfonts}
\usepackage{amssymb}
\usepackage{graphicx}
\usepackage{verbatim}
\usepackage{amsthm}
\usepackage{ragged2e}
\usepackage{amsmath}
\usepackage{float}
\usepackage{caption}
\usepackage{blindtext}
\usepackage{multicol}
\usepackage{subcaption}
\usepackage{amssymb}
\usepackage{mathrsfs}
\usepackage{euscript}
\usepackage{amsthm}
\usepackage{mathtools}
\numberwithin{equation}{section}
\usepackage{bigints}

\usepackage{scalerel,stackengine}
\usepackage{upref,amsthm,amsxtra,exscale}
\usepackage{nameref}
\usepackage{varioref}
\usepackage[colorlinks=true,urlcolor=blue,
citecolor=red,linkcolor=blue,linktocpage,pdfpagelabels,
bookmarksnumbered,bookmarksopen]{hyperref}
\usepackage{cleveref}
\usepackage{cite}
\usepackage{fullpage}
\usepackage[dvipsnames]{xcolor}
\usepackage{lipsum}
\usepackage{setspace}
\providecommand\abstractname{Abstract}

\def\t{\Theta}

\def\bar{\overline}

\renewenvironment{abstract}{%
	\centering\small
	\textbf\abstractname
	\list{}{\leftmargin0.2cm \rightmargin\leftmargin}
	\item\relax
}{%
	\endlist \par\bigskip
}

\advance\oddsidemargin-\textwidth
\evensidemargin=\oddsidemargin

\theoremstyle{plain}
\newtheorem{theorem}{Theorem}[section]
\newtheorem{lemma}[theorem]{Lemma}

\theoremstyle{definition}
\newtheorem{definition}{Definition}[section]

\newtheorem{example}{Example}[section]

\newtheorem{remark}{Remark}

\allowdisplaybreaks

\newcommand{\intav}[1]{\mathchoice {\mathop{\vrule width 6pt height 3 pt depth  -2.5pt
			\kern -8pt \intop}\nolimits_{\kern -6pt#1}} {\mathop{\vrule width
			5pt height 3  pt depth -2.6pt \kern -6pt \intop}\nolimits_{#1}}
	{\mathop{\vrule width 5pt height 3 pt depth -2.6pt \kern -6pt
			\intop}\nolimits_{#1}} {\mathop{\vrule width 5pt height 3 pt depth
			-2.6pt \kern -6pt \intop}\nolimits_{#1}}}

\title{Geometric approaches for improved regularity in fully nonlinear parabolic models}
\author{\it by \smallskip \\ Junior da Silva Bessa, 
	\quad
	Jo\~{a}o Vitor  da Silva
	\quad $\&$  \quad  Mayra Soares
}
\date{\today}

\begin{document}

\maketitle

\begin{abstract}
\noindent In this manuscript,   we derive improved estimates for a class of fully nonlinear parabolic equations with continuous drift and admissible source terms of the form
\[
F(D^2u,x,t)-\partial_tu + \langle B(x,t),Du\rangle
=
f(x,t)\big[(u^+)^\gamma-(u^-)^\gamma\big],
\qquad 0<\gamma<1,
\]
where the coefficient $f$ exhibits $\theta$-H\"older decay near its zero set. By combining a robust blow-up analysis with local derivative estimates for linear equations, we establish improved vanishing estimates at singular nodal points, where the positive and negative phases meet and the solution vanishes to second order.
A central feature of our analysis is the identification of the intrinsic growth exponent
$
\alpha:=\frac{2+\theta}{1-\gamma},
$
which is dictated by the parabolic scaling of the equation and the sublinear homogeneity of the reaction term. Under suitable structural assumptions, solutions attain the corresponding growth rate near singular nodal points, yielding a sharp pointwise regularity estimate. The main novelty lies in the analysis of singular nodal points for sign-changing solutions, where the standard one-phase mechanisms of comparison and nondegeneracy are no longer directly available. 
Our results remain relevant even in linear settings with continuous data, linking to classical free boundary problems arising in mathematical physics and related areas.


\medskip

\end{abstract}

\section{Introduction}

Inspired by a series of semilinear evolution models in divergence form arising in applied mathematics and related topics (see for instance the state-of-the art, subsection \ref{State-of-the-art}, for a complete explanation of such motivations), we consider the following non-divergence form problem:
\begin{equation}\label{P}
 F(D^2u,x,t)-\partial_{t} u + \langle B(x,t), Du\rangle = \mathscr{F}(x,t,u^+,u^{-}) \quad \text{in } Q_1:=B_{1}\times(-1,0].
\end{equation}

Throughout this manuscript, we impose the following structural assumptions:
\begin{itemize}
\item[{\bf ($\rm A_1$)}] ({\bf Uniform Ellipticity})  $F:\mathrm{Sym}(n)\times Q_{1} \to \mathbb{R}$ is a continuous, fully nonlinear and  uniformly elliptic operator with ellipticity constants $0 < \lambda \leq \Lambda$. Precisely, it satisfies
\begin{equation}\label{Unif_parab}
\mathcal{M}^{-}_{\lambda, \Lambda}(\mathrm{Y}) \leq F(\mathrm{X} + \mathrm{Y},x,t) - F(\mathrm{X},x,t) \leq \mathcal{M}^{+}_{\lambda,\Lambda}(\mathrm{Y}),
\end{equation}
for all $(x,t) \in Q_{1}$ and $\mathrm{X}, \mathrm{Y} \in \mathrm{Sym}(n)$ with $\mathrm{Y} \geq 0$. Here $\mathcal{M}_{\lambda,\Lambda}^{\pm}$ denote the Pucci extremal operators:
\[
\mathcal{M}^{+}_{\lambda,\Lambda}(\mathrm{X})=\Lambda \sum_{e_{i}>0}e_{i}+\lambda \sum_{e_{i}<0}e_{i}, \quad
\mathcal{M}^{-}_{\lambda,\Lambda}(\mathrm{X})=\lambda \sum_{e_{i}>0}e_{i}+\Lambda \sum_{e_{i}<0}e_{i},
\]
where $e_{i}(\mathrm{X})$ are the eigenvalues of $\mathrm{X}$. We also normalize $F(\mathrm{O}_n,x,t)=0$.

\item[{\bf ($\rm A_2$)}] ({\bf Regularity of the Data}) 
The coefficients satisfy
\[
B\in \mathcal{C}^{\alpha, \frac{\alpha}{2}}({Q}_{1};\mathbb{R}^{n})
\cap L^{\infty}( Q_1;\mathbb{R}^{n}),
\mbox{
and
}
f \in \mathcal{C}^{0}({Q}_1)\cap L^{\infty}( Q_1)
\] for some exponent $\alpha_0 \in (0, 1)$. Setting $\mathfrak{F}: = f^{-1}(0)$, it is a close subset of $Q_1$, and 
there exist  constants $\mathrm{c}_0>0,$ and 
$ \mathrm{C}_0>0$ and $\theta\in \mathbb{R}_{+}$,  such that
\[ 
\mathrm{c_0}\operatorname{dist}_{p}((x,t),\mathfrak{F})^\theta 
\leq f(x,t) \leq \mathrm{C_0}\operatorname{dist}_{p}((x,t),\mathfrak{F})^\theta \ \mbox{ in }  Q_1,
\]
where $\operatorname{dist}_{p}((x,t),\mathfrak{F}):=\displaystyle\inf_{(y,s)\in \mathfrak{F}}d_{p}((x,t),(y,s))$ is the parabolic distance between the point $(x,t)$ and the set $\mathfrak{F}$.

\item[{\bf ($\rm A_3$)}] ({\bf Parabolic Continuity Condition}) There exists a modulus of continuity $\tau$, such that
\begin{equation}\label{omega}
\frac{|F(\mathrm{X}, x,t)-F(\mathrm{X}, y,s)|}{\|\mathrm{X}\|+1}
\leq
\tau\bigl(d_p((x,t),(y,s))\bigr),
\quad \text{for any } \mathrm{X}\in \mathrm{Sym}(n),
\end{equation}
and $(x,t), (y,s) \in Q_1$, where
\(
d_p\big((x,t),(y,s)\big):=\max\{|x-y|,|t-s|^{\frac{1}{2}}\},
\)
is the parabolic distance. 
Correspondingly,  we  assume $\Phi_{F}(x,t,(x_0,t_0)) \leq
\tau\bigl(d_p((x,t),(y,s))\bigr)$, for  each $(x_0,t_0)\in Q_1$, where 
\[
\Phi_{F}((x,t),(x_0,t_0)) := \sup_{\mathrm{X} \in \mathrm{Sym}(n)} \frac{|F(\mathrm{X},x,t)-F(\mathrm{X},x_0,t_0)|}{\|\mathrm{X}\|+1}, \quad \mbox{for any } \mathrm{X}\in \mathrm{Sym}(n),
\]
is the oscillation function,  which measures the local variation of the coefficients of $F$ near $(x_0,t_0)$ (cf.~\cite{CKS2000}). When $(x_0,t_0)=(0,0)$, we simply write $\Phi_F(x,t)$.

\item[{\bf ($\rm A_4$)}] ({\bf Regularity of the Operator}) For each $(x,t)\in Q_1$, the mapping $F(\cdot,x,t)$ is of class $\mathcal{C}^1(\mathrm{Sym}(n))$. Moreover, there exists a modulus of continuity $\omega$ such that
\[
|\mathrm{D}_{\mathrm{M}}F(\mathrm{X},x,t) - \mathrm{D}_{\mathrm{M}}F(\mathrm{Y},x,t)| \leq \omega(|\mathrm{X}-\mathrm{Y}|),
\]
for all $(x,t)\in Q_1$ and $\mathrm{X},\mathrm{Y}\in \mathrm{Sym}(n)$, where
\[
\mathrm{D}_{\mathrm{M}}F(\mathrm{X}_0,x,t) := \lim_{s \to 0} \frac{F(\mathrm{X}_0+s\mathrm{M},x,t)-F(\mathrm{X}_0,x,t)}{s}
\]
denotes the G{a}teaux derivative with respect to the matrix variable.

\end{itemize}

In our first main result, we are interested in analyzing the regularity of sign-changing viscosity solutions to an evolutionary semilinear two-phase model. We obtain enhanced regularity estimates at singular nodal points by combining a robust blow-up analysis with local derivative estimates for linear equations. 

For this purpose, let us consider, for  $0< \gamma <1$,

\begin{equation}\label{problem2}
F(D^2u,x,t) - \partial_{t}u + \langle B(x,t), Du\rangle =f(x,t)\left[ (u^+)^\gamma - (u^-)^\gamma \right]\quad \text{in } \quad Q_1,
\end{equation}
and define the singular nodal set of order $k\in \mathbb{N}$ as 
\[
\Gamma_k(u) := \{D^\rho_tD^\nu_xu  = 0, \ \forall \ 1\leq 2\rho+|\nu| \leq k \} \cap \partial\{u>0\} \cap \partial\{u<0\}.
\]
Furthermore, we assume that the modulus of continuity given in $(\mathrm{A}_3)$ is $\tau(t) := \mathrm{c}_{\Phi}t^{\zeta}$ for some $\zeta \in (0,1)$ and $\mathrm{c}_{\Phi}>0$ . Let $Q_1 := B_1 \times (-1,0]$, and fix a constant $\mathrm{M}_0 > 0$ such that
$$\max\left\{\Vert{}u\Vert{}_{L^\infty(Q_1)}, [\Phi_F]_{\zeta,Q_1}, \Vert{}B\Vert{}_{\mathcal{C}^{\alpha_0, \frac{\alpha_0}{2}}(\overline{Q}_{1};\mathbb{R}^{n})}, \Vert{}f\Vert{}_{L^{\infty}(Q_{1})}\right\} \leq \mathrm{M}_0,$$
where $[\Phi_F]_{\zeta,Q_1}$ denotes the $\zeta$-Hölder seminorm of the oscillation of the coefficients of $F$.

For any $k \geq 2$ and $\alpha_0 \in (0,1)$, we introduce the function class
$$
\mathcal{G}^{\alpha_0}_k := \left\{ u \in \mathcal{C}^{0}(Q_1) \text{ solving (\ref{P})} : \Vert{}u\Vert{}_{\mathcal{C}^{k+\alpha_0,\frac{k+\alpha_0}{2}}(Q_{1/2})} \leq \mathrm{K}_0(\mathrm{M}_0,\text{data}) \right\},
$$
and consider suitable conditions to guarantee that $\mathcal{G}^{\alpha_0}_2 \neq \emptyset$. For instance, this non-emptiness holds when $F$ is a convex or concave operator, as shown in the seminal work of Wang \cite{Wang90}.

However, under our assumptions, $F$ is not necessarily convex or concave, and the equation contains a lower-order term. Consequently, to ensure $\mathcal{G}^{\alpha_0}_2 \neq \emptyset$ within our framework, one must first establish the $\mathcal{C}^{1+\alpha_0, \frac{1+\alpha_0}{2}}$-regularity estimates provided by \cite[Theorem 1.3]{Wang92}, see also \cite{daSTei17}. Subsequently, by rewriting the equation as

$$
F(D^2u,x,t) - \partial_{t}u = \widetilde{\mathscr{F}}(u, Du, x,t) \quad \text{in } Q_1,
$$
where $\widetilde{\mathscr{F}}(u, Du, x,t) := - \langle B(x,t), Du\rangle + f(x,t)\left[(u^+)^\gamma - (u^-)^\gamma\right]$ is Hölder continuous, one can apply the $\mathcal{C}^{2+\alpha_0, \frac{2+\alpha_0}{2}}$-regularity estimates from \cite[Theorem 1]{SP}.
\medskip

We now state our first main result, which provides improved estimates at the singular nodal set of order $2$.

\begin{theorem}[\bf Improved estimates at singular nodal points]\label{Thm1.3}
Let $u\in\mathcal{G}^{\alpha_0}_2$ be a viscosity solution of problem \eqref{problem2}, and let $(x_{0},t_{0})\in \Gamma_2(u)\cap f^{-1}(0)\cap Q_{1/2}$.
Assume that the structural hypotheses $\rm (A_1)-(A_4)$ hold and that  $\alpha_0-(2+\alpha_0)\gamma<\theta< 1 -3\gamma $. 
Then, for any $\beta$ satisfying
\begin{equation}\label{beta}
0<2+\alpha_{0}<\beta\leq \frac{2+\theta}{1-\gamma},
\end{equation}
there exists $\mathrm{C}>0$, depending only on $n$, $\lambda$, $\Lambda$, $\gamma$, $\|B\|_{\mathcal{C}^{\alpha_0, \frac{\alpha_0}{2}}(\overline{Q}_{1};\mathbb{R}^{n})}$, and $\mathrm{M}_0>0$,
such that
\[
|u(x)|
\le \mathrm{C}\|u\|_{L^\infty(Q_1)}d_p((x,t), (x_{0},t_0))^{\beta},
\quad  \forall \ (x,t)\in Q_{r}(x_{0},t_0),\ \forall \ r\in (0,1/2).
\]
Consequently,
\[
\sup_{Q_{r}(x_{0},t_{0})}|u|
\leq \mathrm{C}\, r^{\beta}
\quad \text{for every } r\in(0,1/2).
\]
\end{theorem}

A central outcome of our analysis is the identification of the intrinsic growth exponent
$
\alpha:=\frac{2+\theta}{1-\gamma},
$
where \(\theta\) measures the rate at which the coefficient \(f\) vanishes. This exponent is dictated by the natural parabolic scaling of the equation and the sublinear homogeneity of the two-phase reaction term. Our estimates show that, under suitable structural assumptions, solutions exhibit the corresponding improved rate of vanishing near singular nodal points. The result is therefore not merely a qualitative regularity statement: it identifies the scaling exponent governing the local behavior of solutions and provides a sharp framework for understanding the interaction between the vanishing of the source term and the nonlinear phase response.

\begin{remark}[\bf Sharpness of the exponent]\label{Sharpness} 
 We observe that  under our assumption  $0\leq\theta <1-3\gamma$, it follows that  $ \alpha \in (2 + \alpha_0,3]$.
The gain \( 2+\alpha_0 < \beta \leq \alpha <3\) reflects an improvement over the natural scaling, obtained via a tangential approximation argument. In fact, if we consider $u \in \mathcal{G}^{\alpha_0}_2$, in view of equation \eqref{problem2}, under the intrinsic scaling 
\[
v_r(x,t) := \frac{u( rx,\,  r^2 t)}{r^{\alpha}},
\]
 we conclude that it satisfies
\[
\frac{1}{r^{\alpha-2}} F\left(r^{\alpha-2} D^2 v_r,\, rx, r^2 t\right)-\partial_{t}v_r
+ r \,\langle B(rx,r^2 t), Dv_r \rangle
= f(rx,r^2t)r^{\alpha\gamma + 2-\alpha } \Big( (v_r^+)^\gamma - (v_r^-)^\gamma \Big),
\]
in the viscosity sense. We obtain the operator $$F_r(X,x,t) : = \frac{1}{r^{\alpha-2}} F\left(r^{\alpha-2} X,\, rx, r^2 t\right),$$ which is also uniformly elliptic, with the same constants that $F$. Defining $
 B_r(x,t): = rB(rx,r^2t),$ $f_{r,\theta}(x,t): = f(rx,r^2)r^{-\theta}$, and since $\alpha(\gamma -1) + 2 + \theta= 0 \iff \alpha =\frac{2+\theta}{1-\gamma}$, it follows that $v_r$ solves
\[
F_r(D^2 v_r, x,t)-\partial_{t} v_r 
+ \langle B_r(x,t), Dv_r \rangle
= f_{r,\theta}(x,t) \big((v_r^+)^\gamma - (v_r^-)^\gamma\big), \mbox{ in }  \ Q_r : = B_r \times (-r^2,0],
\]
with $\mathrm{c_0} 
\leq f_{r,\theta}(x,t) \leq \mathrm{C_0} \ \mbox{ in }  Q_1.$
Thus, the best exponent is, in fact, $\alpha=\frac{2+\theta}{1-\gamma}$.
\end{remark}

{\begin{remark}\label{Caso-thetaplus2}
By following the same approach as in Theorem \ref{Thm1.3}, we may also obtain regularity estimates for viscosity solutions to problem \eqref{P} when the source term $f$ is independent of $u$. In this setting, we establish a class of Schauder-type estimates at points where the source term vanishes and at higher-order singular nodal points. More precisely, if $u$ is a bounded viscosity solution to
\begin{equation}\label{problemwithoutdepofu}
 F(D^2u,x,t)-\partial_{t} u  + \langle B(x,t), Du\rangle =f(x,t) \quad \text{in } Q_1:=B_{1}\times(-1,0]
\end{equation}
and the assumptions $(\mathrm{A}_1)-(\mathrm{A}_4)$ are satisfied, then
$$
\sup_{Q_{r}(x_{0},t_{0})}|u|
\leq \mathrm{C}\, r^{2+\theta}
\quad \text{for every } r\in(0,1/2) \quad \text{and} \quad (x_{0},t_{0}) \in  \Gamma_2(u)\ \cap f^{-1}(0)\cap Q_{1/2}.
$$
\end{remark}

We must recall that pointwise regularity estimates have significant applications in other areas such as free boundary problems  \cite{Lind-Monn2013}, \cite{Lind-Monn2015},  and, \cite[Chapter 7]{PetSHaUral12} (for instance, in the study of the structure of singular sets in obstacle-type problems), and in the analysis of the structure of nodal sets of certain uniformly elliptic PDEs, see Han's work \cite[Theorem 5.1]{Han2000} for an insightful result (see also \cite[Theorem~1.26]{Lian2024} for the parabolic analogue). We must also quote Han's work \cite{Han1998}, which provides a systematic approach to higher-order
parabolic Schauder estimates based on polynomial approximation and
Campanato-type arguments. The main results connect local $L^p$
estimates with pointwise higher regularity under a Dini/Hölder-type
control of the coefficients and the source term. The framework is
well adapted to the intrinsic anisotropic scaling of higher-order
parabolic operators and yields quantitative estimates for both the
solution and its spatial derivatives.

\medskip
 We further obtain nondegeneracy estimates for the corresponding linear-source problem, showing that the order of vanishing of the source is quantitatively reflected in the growth of the solution.
To do so,  we assume that $f$ has a decreasing behavior of  order $\theta$. More precisely, for every local extremal point $(x_{0},t_{0})\in Q_{1}$ and every $r>0$ such that $Q_{r}(x_{0},t_{0})\subset\subset Q_{1}$, there exists a constant $\mathrm{c}_{f}>0$ such that $f$ satisfies
\begin{equation}\label{coninfdhol}
f(x,t)\leq -\mathrm{c}_{f}d_{p}((x,t),(x_0,t_0))^{\theta},\,\, \mbox{ for all } \ (x,t)\in Q_{1}.   
\end{equation}
In this scenario, we prove that, under the non-degeneracy assumption of order $\theta$ on the source term stated in \eqref{coninfdhol}, solutions to \eqref{problemwithoutdepofu} inherit a corresponding non-degeneracy property at extremal points, with order $2+\theta$. This fact is summarized by the following result.

\begin{theorem}[\bf Non-degeneracy]\label{Thm1.2}
Let $u\in \mathcal{C}^{0}(Q_{1})$ be a viscosity solution to problem \eqref{problemwithoutdepofu}. Assume the structural conditions $\rm (A_1)-(A_2)$ and that $f$ satisfies \eqref{coninfdhol}. Then, for every critical point $(x_{0},t_{0})\in Q_{1}$ and every $r>0$ such that $Q_{r}(x_{0},t_{0})\subset\subset Q_{1}$, it follows that  
\[
\sup_{(x,t)\in\partial_{p}Q_{r}(x_{0},t_{0})}|u(x,t)-u(x_{0},t_{0})|\geq \mathrm{C}r^{2+\theta},
\]
for a positive constant $\mathrm{C}$ that depends only on $n$, $\Lambda$, $\|B\|_{L^{\infty}(Q_{1};\mathbb{R}^{n})}$, $\theta$ and $\mathrm{c}_{f}$.
\end{theorem}

Our final result consists of 
an optimal growth estimate for one phase that propagates to the opposite phase near a singular nodal point. This phenomenon is particularly relevant to the two-phase structure of the problem and provides additional information on the coupled behavior of the positive and negative phases.
More precisely, starting from estimates for the positive part (respectively, the negative part), we obtain the same type of estimate for the negative part (respectively, the positive part), a phenomenon known as \textit{flipping estimates}. This is summarized in the following result.
\begin{theorem}[\bf Flipping estimates]\label{Thm1.4}
Let $u\in \mathcal{C}^{0}(Q_{1})$ be a viscosity solution to \eqref{problem2}, and let $(x_{0},t_{0})\in \Gamma_2(u)\cap f^{-1}(0)\cap Q_{1/2}$. Assume that the structural conditions $\rm (A_1)-(A_4)$ are valid with $\tau(t):=\mathrm{c}_{\Phi}t^{\zeta}$ in the condition $\rm (A_3)$ for $\zeta \in \left[\frac{2+\theta}{1-\gamma}-2, 1\right]$, and there exists $r_{0}>0$ such that $Q_{r_{0}}(x_{0},t_{0})\subset\subset Q_{1}$. If
\[
\sup_{Q_r(x_{0},t_{0})} u^{+} \leq \mathrm{C}_1 r^{\frac{2+\theta}{1-\gamma}} \quad (\text{resp. } \sup_{Q_r(x_{0},t_{0})} u^- \leq \mathrm{C}_1 r^{\frac{2+\theta}{1-\gamma}}) \quad \text{for all } r \in (0,r_0],
\]
then there exists \(\mathrm{C}_1 > 0\) such that
\[
\sup_{Q_r(x_{0},t_{0})} u^- \leq \mathrm{C}_2 r^{\frac{2+\theta}{1-\gamma}} \quad \big(\text{resp. } \sup_{Q_r(x_0,t_{0})} u^+ \leq \mathrm{C}_2 r^{\frac{2+\theta}{1-\gamma}}\big) \quad \text{for all } r \in (0,r_0].
\]
\end{theorem}

An important feature of our problem is that it does not rely on a variational structure (cf. \cite{Fot-Shah2017} and \cite{SoaTer18}). This is particularly relevant for the class of equations considered here, since the operator is fully nonlinear, the coefficients may depend on space and time, and a lower-order drift term is present. Moreover, the solutions are allowed to change sign. At two-phase nodal points, the maximum principle and standard one-phase arguments do not provide the required control, since such points are not, in general, local extrema of the solution. Our approach circumvents these difficulties by exploiting intrinsic scaling, compactness, tangential approximation, and quantitative derivative estimates for the limiting linear problem.

To the best of our knowledge, the sharp pointwise analysis of two-phase semilinear parabolic equations with sublinear power reactions in the fully nonlinear nondivergence setting considered here has received very limited attention. Our work is intended to fill this gap by providing both quantitative regularity estimates and a geometric mechanism explaining their scaling. We believe that these results offer a meaningful contribution to the regularity theory of nonlinear parabolic equations and to the study of free-boundary and nodal phenomena (cf. \cite{Choe-Weiss2003})

\subsection{The heuristic behind our regularity estimates}

The key idea behind the derivation of the a priori estimates for the semilinear evolution model \eqref{problem2}, that is, the refined estimates stated in Theorem \ref{Thm1.3},  is to access these improved regularity estimates through a geometric tangential method, following the framework introduced in (cf. \cite{SP}).

By way of explanation, consider a fully nonlinear elliptic operator \( F: \text{Sym}(n) \to \mathbb{R} \) such that $F(\mathrm{O}_n) = 0$ (this is not a restrictive assumption). Then, the family of elliptic scaling functions defined by
\[
\mathcal{T}_{\varrho}(\mathrm{X}) := \frac{1}{\varrho} F(\varrho \mathrm{X}), \quad \text{for } \varrho > 0,
\]
constitutes a continuous family of operators that preserves the ellipticity constants of the original equation (i.e., $0< \lambda\leq \Lambda$). Hence, if \( F: \text{Sym}(n) \to \mathbb{R}\) is differentiable (for instance, at the origin, recalling the normalization \( F(\mathrm{O}_n) = 0 \)), then we have
\[
\displaystyle \lim_{\varrho \to 0} \mathcal{T}_{\varrho}(\mathrm{X}) = \lim_{\varrho \to 0} \frac{F(\varrho \mathrm{X}+\mathrm{O}_n)-F(\mathrm{O}_n)}{\varrho} = \mathrm{D}_{\mathrm{X}}F(\mathrm{O}_n) = \sum_{i,j=1}^{n}\frac{ \partial F}{\partial \mathrm{X}_{ij}}(\mathrm{O}_n) \mathrm{X}_{ij}.
\]
In other words, the linear, uniformly elliptic operator \(\displaystyle \mathrm{X} \mapsto \sum_{i,j=1}^{n}\frac{ \partial F}{\partial \mathrm{X}_{ij}}(\mathrm{O}_n) \mathrm{X}_{ij}  = \mathrm{tr}(\mathfrak{A}_0\mathrm{X})\) represents the \textbf{geometric tangential equation} associated with \(\mathcal{T}_{\varrho} \) in the limiting configuration as \( \varrho \to 0 \). 

In this setting, if \( u: \mathrm{Q}_1 \to \mathbb{R} \) is a solution to an equation involving the original operator \( F: \text{Sym}(n) \to \mathbb{R}\), i.e.,
$$
F(D^2 u) - \partial_t u + \langle \mathfrak{B}(x), Du\rangle = f(x, t, u) \quad \text{in} \quad \mathrm{Q}_1,
$$
then the scaled function \( u^{\mathfrak{a}}_{\varrho}(x, t) := \frac{1}{\varrho^{2+\alpha}\mathfrak{a}} u(\varrho x, \varrho^2 t) \), which is a blow-up at singular nodal points, satisfies a corresponding equation, namely,
$$
 \mathcal{T}_{\varrho^{\alpha} \mathfrak{a}}(D^2 u^{\mathfrak{a}}_{\varrho}(x, t)) -\partial_t u^{\mathfrak{a}}_{\varrho} + \langle \mathfrak{B}_{\varrho}^{\mathfrak{a}}(x, t), D u^{\mathfrak{a}}_{\varrho}(x, t)\rangle = \frac{1}{\varrho^{\alpha} \mathfrak{a}} f(\varrho x, \varrho^2 t, u(\varrho x, \varrho^2 t)) := f_{\varrho}^{\mathfrak{a}}(x, t, u^{\mathfrak{a}}_{\varrho}) \quad \text{in} \quad \mathrm{Q}_{1/\varrho}, 
$$
where $\mathfrak{B}_{\varrho}^{\mathfrak{a}}(x, t) = \frac{\varrho}{a} \mathfrak{B}(\varrho x, \varrho^2 t)$, and  $\mathfrak{a}>0$ is chosen in such a way $\|f_{\varrho}^{\mathfrak{a}}\|_{\infty, \mathrm{Q}_1} = \text{o}(1)$ and $\|\mathfrak{B}_{\varrho}^{\mathfrak{a}}\|_{\infty, Q_1} = \text{o}(1)$ as $\varrho \to 0$.
Moreover, if one establishes that the norm \( \| u^{\mathfrak{a}}_{\varrho}\|_{\mathcal{C}^{2+\alpha, \frac{2+\alpha}{2}}(Q_{r})} \leq \mathrm{C}_0 \) is bounded, then \( u^{\mathfrak{a}}_{\varrho} \) is a normalized solution (in the $(2+\alpha)$-setting) to the ``\( \varrho \)-scaled equation.''

This approach enables us to exploit regularity results (e.g., local derivative estimates) available for the linear tangential equation via compactness and stability arguments. In particular, we obtain the following:
$$
\begin{array}{ccc}
 \left\{ 
 \begin{array}{ccl}
 \mathcal{T}_{\varrho^{\alpha} \mathfrak{a}}(\mathrm{X}) \to  \mathrm{tr}(\mathfrak{A}_0 \mathrm{X}) & \text{as}& \varrho \to 0    \\
 \mathfrak{B}_{\varrho}^{\mathfrak{a}} \to 0 & \text{as} & \varrho \to 0\\
 f_{\varrho}^{\mathfrak{a}} \to 0 & \text{as}& \varrho \to 0\\
 u^{\mathfrak{a}}_{\varrho} \to \mathbf{U} & \text{as}& \varrho \to 0
 \end{array}
 \right.& \Rightarrow & \partial_t \mathbf{U} -\mathrm{tr}(\mathfrak{A}_0 D^2 \mathbf{U}) = 0 \quad \text{in} \quad \mathbb{R}^n\times (-\infty, 0].
\end{array}
$$
{Moreover, one has
\[
\|\mathbf{U}\|_{L^\infty(Q_r)} \leq \mathrm{C}r^{2+\alpha},  \quad \text{and} \quad D^\rho_tD^\nu_x \mathbf{U}(0, 0)  = 0, \ \forall \ 0\leq 2\rho+|\nu| \leq 2.  
\]
Additionally, 
\[
\quad \|D^\rho_tD^\nu_x \mathbf{U}\|_{L^\infty(Q_r)} \leq \mathrm{C_0}r^{2+\alpha-(k+1)} \quad \forall \,\,\, 2\rho+|\nu| \geq k+1 \geq 3
\]
In this setting, the available derivative estimates combined with a sort of Liouville-type results for the linear limiting profile allow us to derive enhanced regularity estimates for the original profile via a blow-up argument along singular nodal points.
We refer the reader to the following insightful recent surveys \cite{SP}, \cite[Ch.~5]{João}, \cite{Teixeira2016}, and \cite{Teixeira2020} for developments in geometric tangential methods and their applications to nonlinear PDEs and related topics.

\bigskip

Finally, to obtain flipping estimates, i.e., Theorem \ref{Thm1.4}, we have the following geometric  strategy. If $(x_{0}, t_0) \in\Gamma_2(u)\cap f^{-1}(0)\cap Q_{1/2}$,  and
 there exists  $r_{0}  >0$ such  that 
$$\sup\limits_{\mathrm{Q}_{r}( x_{_{0}}, t_0)} u^{\mp} \leq \mathrm{C}_{0} r^{\frac{2+\theta}{1-\gamma}} \ \mbox{ for all } \ \ r\in ( 0,r_{0}],
$$
by an iterative argument there exists a positive  constant $\mathrm{C}_0$ such that
$$
\| u\| _{L^{\infty }(\mathrm{Q}_{2^{-( j+1)}})} \leq \max\left\{\frac{\mathrm{C}_0}{2^{\frac{2+\theta}{1-\gamma}( j+1)}} ,\frac{\| u\| _{L^{\infty }(\mathrm{Q}_{2^{-j}})}}{2^{\frac{2+\theta}{1-\gamma}}}\right\}.
$$
As in the previous estimates, this geometric decay follows from a refined blow-up argument combined with geometric tangential methods. These techniques capitalize an improvement-of-flatness property of a one-phase limiting profile associated with a linear operator. Consequently, there exists $\mathrm{C}_{1} > 0$ such that
  $$
\sup\limits_{\mathrm{Q}_{r}( x_{_{0}}, t_0)} u^{\pm} \leq \mathrm{C}_{1}(\text{universal}) r^{\frac{2+\theta}{1-\gamma}}\ \mbox{ for all } \ \ r\in ( 0,r_{0}].
$$

\subsection{Semilinear parabolic models and related free boundaries}\label{State-of-the-art}

\noindent
A central theme in the modern theory of evolution equations concerns the qualitative behavior of solutions, with particular emphasis on the regularity of bounded solutions. Establishing sharp regularity estimates, including Hölder continuity, gradient bounds, and higher-order smoothness, is vital for understanding stability, uniqueness, and long-time dynamics. These concepts connect deeply to nonlinear potential theory, geometric methods, and the structure of degenerate and singular operators.

A prominent class of nonlinear evolution equations is given by the \emph{reaction--diffusion models}, introduced in the twentieth century to describe a broad range of processes in Physics and Biology. Their applications span Mechanics, Materials Science, Biophysics, and Ecology, where they capture diffusion, transport, and nonlinear interaction effects.
In general, such phenomena are modeled by parabolic equations in divergence form
\begin{equation}\label{eq:general_parabolic}
\partial_t u = \operatorname{div} \mathfrak{A}(\nabla u, u, x,t) + \mathfrak{B}(\nabla u, u, x,t),
\end{equation}
where $\operatorname{div} \mathfrak{A}$ describes diffusion, while the lower-order term $\mathfrak{B}$ encodes reaction, absorption, or convection. 
A canonical prototype is the \emph{semilinear heat equation}
\begin{equation}\label{eq:semilinear_heat}
\partial_t u = \Delta u + f(u),
\end{equation}
which serves as a fundamental model for nonlinear diffusion with source or absorption terms (see, e.g.,~\cite{BE1989,Smoller1983}).
For concreteness, one may interpret $u \geq 0$ as a temperature (or concentration), where the Laplacian models diffusion and the nonlinearity $f(u)$ represents internal sources, sinks, or phase transitions. These equations are typically posed in a domain $\Omega \subset \mathbb{R}^n$, supplemented with an initial condition
\begin{equation}\label{eq:initial_data}
u(x,0) = u_0(x),
\end{equation}
and appropriate boundary conditions.

\medskip

A key aspect in the analysis of semilinear parabolic equations is the interplay between diffusion and nonlinear effects, which leads to phenomena such as pattern formation, blow-up, and the emergence of \emph{free boundaries}. In particular, nonlinear absorption or degeneracy may produce interfaces separating regions where the solution is positive from those where it is negative or vanishes, naturally leading to free boundary problems. The analysis of these interfaces, together with the regularity of solutions and their gradients near the free boundary, is a central topic in current research, bridging parabolic regularity theory with geometric and variational techniques.

By way of illustration in the one-phase scenario, recently, Audrito and Sanz-Perela \cite{AS-P2024} introduced a novel elliptic regularization framework to establish the existence of strong solutions for a class of semilinear parabolic free boundary problems
\begin{equation}\label{eq}
\partial_t u - \Delta u = -f_\gamma(u) \quad \text{in } Q := \mathbb{R}^n \times (0,\infty), \qquad u|\{t=0\} = u_0 \quad \text{in } \mathbb{R}^n,
\end{equation}
where $n \geq 1$, $\gamma \in [1,2)$, $u_0 \geq 0$, and $f_\gamma(u) := \gamma \chi_{\{u>0\}} u^{\gamma-1}$.
Problems of this type arise in chemical engineering and in the transport of thermal energy in plasmas (see~\cite{Martinson1976}). 

The mathematical analysis of such equations and their associated free boundaries was initiated by Caffarelli~\cite{CPS2004}  in the case $\gamma=1$, where \eqref{eq} reduces to a variant of the classical Stefan problem. We also must highlight the contribution of Weiss~\cite{Weiss2000} (see also Choe and Weiss~\cite{Choe-Weiss2003}), who established the classification of blow-up limits and fine regularity properties of the free boundary, together with sharp estimates on the parabolic Hausdorff dimension of $\partial\{u>0\}$.

\medskip

The literature on two-phase problems for quasilinear models (stationary scenario) remains relatively scarce, largely due to the lack of fundamental analytical tools, such as monotonicity formulas available in the classical Laplacian setting, which substantially streamline the analysis. Shahgholian and Fotouhi \cite{Fot-Shah2017} investigated the semilinear model
\begin{eqnarray}\label{eqsemilinear}
\Delta u = \lambda_{+}(x)(u^{+})^{q-1} - \lambda_{-}(x)(u^{-})^{q-1} \quad \text{in} \quad \mathrm{B}_1,
\end{eqnarray}
with emphasis on the regularity of solutions and the structure of the free boundary \(\partial\{\pm u>0\}\). Assuming that \(\lambda_{\pm}\) are Lipschitz continuous and \(1 < q < 2\), they derived local regularity results for both the solutions and their associated free boundaries.

In a related direction, Soave and Terracini \cite{SoaTer18} studied nodal sets for a particular instance of \eqref{eqsemilinear}, corresponding to constant positive \(\lambda_{\pm}\) and \(1 \leq q < 2\). They established finiteness of the vanishing order at every point, characterized the full spectrum of possible orders, proved a weak non-degeneracy property, and obtained regularity of the nodal set along with a partial stratification result.

\bigskip

In \cite{Shahgholian-Uraltseva-Weiss2017} the authors studied a parabolic two-phase obstacle problem 
$$
\Delta u=\partial_t u+\lambda_{+}\chi_{\{u>0\}}-\lambda_{-}\chi_{\{u<0\}},
$$
and established optimal regularity for the free boundary. In particular, near branch points, they proved that the free boundary is locally represented by two Lipschitz graphs that are $C^1$ in the spatial variables. A key ingredient was a directional monotonicity principle, which yielded space and space-time monotonicity for solutions sufficiently close to the one-dimensional profile, combined with a monotonicity formula to classify global solutions. They also constructed a counterexample showing that the Lipschitz regularity is, in general, optimal.

In the same direction, in \cite{Lindgren-Prajapat2012} Lindgren and Prajapat studied a two-phase parabolic free boundary problem 
$$
\mathrm{H}[u] = \Delta u-\partial_tu = \lambda_1(x, t)\chi_{\{u>0\}}-\lambda_2(x, t)\chi_{\{u<0\}} \quad \text{in} \quad  Q_1
$$
under weaker regularity assumptions on the coefficients, without imposing the usual Lipschitz continuity of the right-hand side. They obtained optimal regularity of solutions, namely $C_x^{1,1} \cap C_t^{0,1}$, by means of scaling arguments, replacing the monotonicity-formula approach commonly used in this setting.

\bigskip

Inspired by the stationary setting, a semilinear two-phase  classical
strong-absorption problem with
power-type nonlinearities of the form
\[
    f(u)=(u^+)^p-(u^-)^p,
    \qquad
    u^+:=\max\{u,0\},
    \quad
    u^-:=\max\{-u,0\},
\]
arise naturally in reaction--diffusion models with
phase-dependent nonlinear responses. The corresponding elliptic problem has been
studied extensively from the viewpoint of the regularity of both the
solution and the free boundary
\cite{Fot-Shah2017}.
In the parabolic setting, the one-phase counterpart
\[
    \partial_tu-\Delta u=-u^p\chi_{\{u>0\}}
\]
has been studied in detail, including optimal regularity,
nondegeneracy, and free-boundary properties by 
Choe-Weiss \cite{Choe-Weiss2003}. Moreover, in such a manuscript the analysis of such
interfaces combines parabolic regularity, blow-up methods, monotonicity
formulae, and geometric-measure-theoretic techniques.
This work provides  a natural framework for studying strong absorption, moving
interfaces, and free-boundary phenomena; closely related models also
arise in combustion and phase-transition problems, see also
\cite{Weiss2003}.

Regarding the two-phases model, given a domain $\Omega\subset\mathbb R^n$ and an initial
datum $u_0$, our model  \eqref{P} embraces the particular problem
\[
    \Delta u-\partial_tu
    =(u^+)^p-(u^-)^p
    \qquad\text{in }\Omega\times(0,T),
\]
with $0<p<1$. In the positive phase
$\{u>0\}$, the equation reduces to
   $ \partial_tu-\Delta u=-(u^{+})^p,$
while, in the negative phase $\{u<0\}$, the function satisfies
   $ \partial_t u-\Delta u= (u^{-})^p.$
Thus, each phase evolves according to a heat equation with a
sublinear absorption term. Since $s\mapsto s^p$ is non-Lipschitz at
$s=0$ for $0<p<1$, the zero set may develop a nontrivial free boundary (cf. \cite{Choe-Weiss2003}).
The resulting interface between the positive and negative phases is,
therefore, an intrinsic unknown of the problem. }

\bigskip

Motivated by the stationary theory developed by Fotouhi and Shahgholian~\cite{Fot-Shah2017} and  Soave and Terracini \cite{SoaTer18}, as well as two-phase evolution models of obstacle-type addressed by Shahgholian \textit{et al} \cite{Shahgholian-Uraltseva-Weiss2017}, and Lindgren and Prajapat \cite{Lindgren-Prajapat2012}, our manuscript investigates the behavior of two-phase solutions at singular nodal points for the corresponding semilinear parabolic problem. Building on this framework while introducing genuinely parabolic techniques, we obtain new results on the local behavior and sharp regularity of solutions near such singular free-boundary points. Our approach is, to the best of our knowledge, new and provides a flexible framework that may be adapted to other nonlinear evolution models.

\bigskip

\subsection{\bf Main obstacles to overcome} The study of semilinear parabolic equations with \emph{sign-changing solutions} presents substantial additional difficulties compared to the one-phase setting (cf. \cite{DaSO2019} and \cite{DaSOS2018} for enlightening examples). In particular, classical techniques for local regularity estimates are no longer applicable and must be replaced by alternative approaches.

A major obstruction is the limited applicability of the maximum principle. In the one-phase setting, typically associated with nonnegative solutions, it is a key tool for deriving optimal estimates (see, e.g.,~\cite{DaSO2019}). However, in the present framework, two issues arise. First, the validity of the maximum principle in the parabolic setting often requires additional structural assumptions, such as monotonicity in time (cf.~\cite[Theorem 6.1]{DaSO2019}), which are not satisfied here. Second, in contrast to the one-phase case, free boundary points are no longer local minima of the solution.

To overcome these difficulties, we adopt an alternative strategy inspired by~\cite[Section 3]{L-S-R-O2025}, based on the analysis of solution growth under the natural scaling of the equation, combined with refined derivative estimates for associated linear profiles. This approach bypasses comparison principles and instead exploits intrinsic scaling and compactness methods.

As a result, we obtain sharp regularity estimates at \emph{singular nodal points}, namely points where the solution vanishes together with some of its derivatives, typically up to second order. These results reveal new structural features of the model in the two-phase setting, which are absent in the classical one-phase theory.

To the best of our mathematical knowledge, there are not further regularity results in the scenario of two-phase semilinear models like  \eqref{P} (including the linear  setting), and such an absence of investigation was one of the main impetus of developing a deep analysis in singular nodal points of such evolution models.   

\section{Preliminaries}

In this section, we introduce the notation that will be used throughout the remainder of this manuscript. 
For $x_{0}\in\mathbb{R}^{n}$ and $r>0$, we denote by $B_{r}(x_{0})$ the open ball centered at $x_{0}$ with radius $r$, namely,
\[
B_{r}(x_{0})=\{x\in\mathbb{R}^{n}\,:\,|x-x_{0}|<r\}.
\]
When $x_{0}=0$, we simply write $B_{r}:=B_{r}(0)$.
Given $(x_{0},t_{0})\in\mathbb{R}^{n}\times\mathbb{R}$ and $r>0$, we define the lower parabolic cylinder
\[
Q_r(x_0,t_0) := B_r(x_0) \times (t_0 - r^2, t_0].
\]
When $x_0=0$ and $t_0=0$, we omit the center and write $Q_r:=Q_r(0,0)$.

We also introduce the parabolic distance between two points $P_{1}=(x,t)$ and $P_{2}=(y,s)$ in $\mathbb{R}^{n}\times\mathbb{R}$, defined by
\[
d_p\big((x,t),(y,s)\big):=\max\left\{|x-y|,|t-s|^{\frac{1}{2}}\right\}.
\]

In the parabolic framework, for a function
$u:\Omega\subset\mathbb{R}^{n}\times\mathbb{R}\to\mathbb{R}$, we denote its mixed derivatives by $D_t^\rho D_x^\nu u$, where \(\nu\) is a  multi-index corresponding to the spatial variable and \(\rho \in \mathbb{N}\cup\{0\}\) corresponds to the temporal variable. In particular, when \(|\nu|=0\) and \(\rho=1\), we simply write $\partial_t u$ for the first-order time derivative. For simplicity, we write \(Du\) and \(D^2u\) to denote the spatial gradient and the spatial Hessian of \(u\), respectively.

On the other hand, we will require the concept of H\"older spaces (see e.g. \cite{IS13}). Let $\alpha\in(0,1]$ and let $Q$ be a domain in $\mathbb{R}^{n}\times\mathbb{R}$. We say that $u\in \mathcal{C}^{\alpha,\frac{\alpha}{2}}(\overline{Q})$ if
\[
\|u\|_{\mathcal{C}^{\alpha,\frac{\alpha}{2}}(\overline{Q})}
:= \|u\|_{L^\infty(\overline{Q})} + [u]_{\alpha,\overline{Q}}<+\infty,
\]
where
\[
[u]_{\alpha,\overline{Q}}
:= \sup_{\substack{(x,t),(y,s)\in \overline{Q}\\(x,t)\neq(y,s)}}
\frac{|u(x,t)-u(y,s)|}{d_{p}((x,t),(y,s))^{\alpha}}.
\]

We also introduce the following notation:
\[
[u]_{1+\alpha,\overline{Q}}
:=\sup_{\substack{(x,t),(x,s)\in \overline{Q}\\t\neq s}}
\frac{|u(x,t)-u(x,s)|}{|t-s|^{\frac{1+\alpha}{2}}}.
\]
With this definition at hand, we can introduce the higher order parabolic H\"older spaces. We define the H\"older spaces $\mathcal{C}^{1+\alpha,\frac{1+\alpha}{2}}(\overline{Q})$ and $\mathcal{C}^{2+\alpha,1+\frac{\alpha}{2}}(\overline{Q})$ as the sets of all functions $u$ such that
\begin{eqnarray*}
\|u\|_{\mathcal{C}^{1+\alpha,\frac{1+\alpha}{2}}(\overline{Q})}
&:=& 
\|u\|_{L^{\infty}(\overline{Q})}
+\|Du\|_{L^{\infty}(\overline{Q})}
+[u]_{\mathcal{C}^{1+\alpha,\frac{1+\alpha}{2}}(\overline{Q})}
< +\infty,
\\
\|u\|_{\mathcal{C}^{2+\alpha,1+\frac{\alpha}{2}}(\overline{Q})}
&:=&
\|u\|_{L^{\infty}(\overline{Q})}
+\|Du\|_{L^{\infty}(\overline{Q})}
+\|\partial_{t}u\|_{L^{\infty}(\overline{Q})}
+\|D^{2}u\|_{L^{\infty}(\overline{Q})}+
 [u]_{\mathcal{C}^{2+\alpha,\frac{2+\alpha}{2}}(\overline{Q})}
< +\infty,
\end{eqnarray*}
where 
\begin{equation*}
 [u]_{\mathcal{C}^{1+\alpha,\frac{1+\alpha}{2}}(\overline{Q})} :=   [Du]_{\alpha,\overline{Q}}
+[u]_{1+\alpha,\overline{Q}}
\end{equation*}
and
\begin{equation*}
 [u]_{\mathcal{C}^{2+\alpha,\frac{2+\alpha}{2}}(\overline{Q})} :=    [\partial_{t}u]_{\alpha,\overline{Q}}
+[D^{2}u]_{\alpha,\overline{Q}}
+[Du]_{1+\alpha,\overline{Q}}.
\end{equation*}

Finally, we introduce the notion of the solution adopted throughout this paper. Consider the model equation
\begin{equation}\label{2.1}
\partial_{t}u-F(D^{2}u,x,t)+\langle B(x,t),Du\rangle=f(x,t,u)
\quad \text{in } Q\subset \mathbb{R}^{n}\times\mathbb{R},    
\end{equation}
where $Q$ is a domain, $B:Q\to\mathbb{R}^{n}$ is a continuous vector field, and 
$f:Q\times\mathbb{R}\to\mathbb{R}$ is a continuous function.

\begin{definition}[{\bf Viscosity solutions}]
Let $Q\subset \mathbb{R}^{n}\times\mathbb{R}$ be an open set. A function
$u\in C^{0}(Q)$ is called a \emph{viscosity subsolution} of

$$
F(D^{2}u,x,t)-\partial_tu+\langle B(x,t),Du\rangle
=
f(x,t)\big[(u^{+})^{\gamma}-(u^{-})^{\gamma}\big]
\quad\text{in }Q,
$$
if, for every $(x_{0},t_{0})\in Q$ and every test function
$\varphi\in C^{2,1}(Q)$ such that $u-\varphi$ attains a local maximum at
$(x_{0},t_{0})$, one has

$$
F(D^{2}\varphi(x_{0},t_{0}),x_{0},t_{0})
-\partial_t\varphi(x_{0},t_{0})
+\big\langle B(x_{0},t_{0}),D\varphi(x_{0},t_{0})\big\rangle
\geq
f(x_{0},t_{0})
\Big[
\big(\varphi^{+}(x_{0},t_{0})\big)^{\gamma}
-
\big(\varphi^{-}(x_{0},t_{0})\big)^{\gamma}
\Big].
$$

Similarly, $u\in C^{0}(Q)$ is called a \emph{viscosity supersolution} if, for every
$(x_{0},t_{0})\in Q$ and every $\varphi\in C^{2,1}(Q)$ such that $u-\varphi$
attains a local minimum at $(x_{0},t_{0})$, one has

$$
F(D^{2}\varphi(x_{0},t_{0}),x_{0},t_{0})
-\partial_t\varphi(x_{0},t_{0})
+\big\langle B(x_{0},t_{0}),D\varphi(x_{0},t_{0})\big\rangle
\leq
f(x_{0},t_{0})
\Big[
\big(\varphi^{+}(x_{0},t_{0})\big)^{\gamma}
-
\big(\varphi^{-}(x_{0},t_{0})\big)^{\gamma}
\Big].
$$

A function $u\in C^{0}(Q)$ is called a \emph{viscosity solution} of the equation in $Q$
if it is both a viscosity subsolution and a viscosity supersolution.
\end{definition}

The following well-known result, namely the Comparison Principle, will be needed in the proof of Theorem \ref{Thm1.2}. We refer the reader, for instance, to \cite[Theorem 8.2]{CIL92} for a proof.

\begin{theorem}[\bf Comparison Principle]\label{CompPrinc}
Assume the $F:\text{Sym}(n)\times Q_{r}\to \mathbb{R}$ is a continuous function and satisfies the structural condition $(\mathrm{A}_1)$ and let \(B\in \mathcal{C}^{0}(Q_{r},\mathbb{R}^{n})\) be a vector field, and \(f,g\) be continuous functions in \(Q_{r}\). If \(u\) and \(v\) are respectively a viscosity subsolution and a viscosity supersolution of \eqref{problemwithoutdepofu}, and \(u\leq v\) on \(\partial_{p}Q_{r}\), then \(u\leq v\) in \(Q_{r}\).
\end{theorem}

We conclude this section with the following stability result, whose proof follows along the same lines as in \cite[Theorem 6.1]{CKS2000}.

\begin{lemma}[{\bf Stability Lemma}]\label{Est}
Consider $\{\Omega_j\}_{j \in \mathbb{N}}$ be an increasing sequence of open sets in $\mathbb{R}^{n} \times \mathbb{R}$ such that $\Omega_j \subset \Omega_{j+1}$ and define $\Omega := \bigcup_{k=1}^{\infty} \Omega_j$. Suppose that $F$, $F_j$ are $(\lambda, \Lambda)$-parabolic  operators. Assume $f \in \mathcal{C}^{0}(\Omega; \mathbb{R})$, $f_j \in \mathcal{C}^{0}(\Omega_j;\mathbb{R})$, $B_{j}\in \mathcal{C}^{0}(\Omega_{j};\mathbb{R}^{n})$, $B\in \mathcal{C}^{0}(\Omega;\mathbb{R}^{n})$, and let $u_j \in \mathcal{C}^0(\Omega_j)$ be a viscosity subsolution (resp. supersolution) of 
\[
\partial_{t}u_{j}-F_j(D^2 u_j,x,t)+ \langle B_{j}(x,t), Du_j\rangle= f_j(x,t,u_{j}) \quad \text{in} \quad \Omega_j,
\]
Suppose that $u_j \to u_{\infty}$ locally uniformly in $\Omega$, $F_{kj}\to F$ locally uniformly in $\text{Sym}(n)$, $f_{j}\to f$ and $B_{j}\to B$. Then $u_\infty$ is a viscosity subsolution (resp. supersolution) of
\[
\partial_{t}u_{\infty}-F(D^2 u_{\infty},x, t)+\langle B(x,t), Du_{\infty}\rangle= f(x,t,u_{\infty}) \quad \text{in} \quad \Omega.
\]
\end{lemma}

\section{ Proof of the Theorems \ref{Thm1.3}, \ref{Thm1.2} and \ref{Thm1.4}}

In this final section, we focus on the results concerning the semilinear model \eqref{problem2}. We begin by presenting the proof of Theorem \ref{Thm1.3}.

\begin{proof}[\bf Proof of Theorem \ref{Thm1.3}] If $u\equiv 0$, the result is trivial, hence we can suppose that $\|u\|_{L^{\infty}(Q_{1})}>0$.
By the structural assumptions on $F$ and the standard parabolic regularity, we have
\begin{equation}\label{eq:bound}
\|u\|_{\mathcal{C}^{2+\alpha_0,\frac{2+\alpha_0}{2}}(Q_{1/2})} \leq \mathrm{K}_0,
\end{equation}
for each $u\in \mathcal{G}^{\alpha_0}_2$. Without loss of generality, we assume again $(x_0,t_0) = (0,0)\in \Gamma_2(u)\cap f^{-1}(0)\cap Q_{1/2}$ and $\|u\|_{L^{\infty}(Q_{1})}=1$, hence, it suffices to prove
\[
|u(x,t)| \leq \mathrm{C}d_p((x,t)(0,0))^{\beta},
\]
for $\beta$ given in \eqref{beta}. We argue by contradiction. Assuming that this estimate fails, there would exist a sequence
$u_j \in \mathcal{G}^{\alpha_0}_2$, and points $(x_j,t_j) \to (0,0)$ such that
\begin{equation}\label{eq:contradiction}
|u_j(x_j,t_j)| > j d_p((x_j,t_j),(0,0))^{\beta}.
\end{equation}
In addition, setting the oscillation quantity,
\[
\phi_j(\rho) := \sup_{\rho<r<1/2} r^{2+\alpha_0-\beta}[u_j]_{\mathcal{C}^{2+\alpha_0,\frac{2+\alpha_0}{2}}(Q_{r})}, \]
we claim that it should yield 
\(
\displaystyle\lim_{j\to \infty}\lim_{\rho \to 0} \phi_j(\rho) = + \infty,
\)
and thus, there would exist a sequence $\{r_j\}$ of positive numbers, such that $r_j \to 0$ as $j\to \infty.$ 

In fact,
by definition,
\[
\lim_{\rho \to 0} \phi_j(\rho)
= \sup_{0<r<1/2} r^{2+\alpha_0-\beta} [u_j]_{\mathcal{C}^{2+\alpha_0,\frac{2+\alpha_0}{2}}(Q_r)}.
\]
Since $(0,0)\in \Gamma_2(u_j)$, it follows that
\begin{align}\label{estwithuj}
|u_{j}(x_{j},t_{j})|&\leq\left|u_{j}(x_{j},t_{j})-u_{j}(0,t_{j})-Du_{j}(0,t_{j})\cdot x_{j}-\frac{1}{2}x_{j}^{T}D^{2}u_{j}(0,t_{j})x_{j}\right|\nonumber\\
&+|Du_{j}(0,t_{j})-Du_{j}(0,0)||x_{j}|
+\frac{1}{2}|x_{j}^{T}(D^{2}u_{j}(0,t_{j})-D^{2}u_{j}(0,0))x_{j}|\nonumber\\
&+|u_{j}(0,t_{j})-u_{j}(0,0)-\partial_{t} u_{j}(0,0)t_{j}|\nonumber\\
&\leq[D^{2}u_{j}]_{\alpha_{0},\overline{Q^{+}_{r_j}}}|x_{j}|^{2}(|x_{j}|^{\alpha_{0}}+|t_{j}|^{\frac{\alpha_{0}}{2}})+[\partial_{t}u_{j}]_{\alpha_{0},\overline{Q^{+}_{r_j}}}|t_{j}|^{1+\frac{\alpha_{0}}{2}}+[Du_{j}]_{1+\alpha_{0},\overline{Q^{+}_{r_j}}}|x_{j}||t_{j}|^{\frac{1+\alpha_{0}}{2}}\nonumber\\
&\leq[u_{j}]_{\mathcal{C}^{2+\alpha_{0},\frac{2+\alpha_{0}}{2}}(Q_{r_j})}d_{p}((x_{j},t_{j}),(0,0))^{2+\alpha_{0}}.
\end{align}
 From \eqref{eq:contradiction} and \eqref{estwithuj}, 
\begin{align}\label{eq:contradiction2}
j < \frac{|u_j(x_j,t_{j})|}{d_p((x_j,t_j),(0,0))^{\beta}}
&\leq \frac{[u_j]_{\mathcal{C}^{2+\alpha_0,\frac{2+\alpha_0}{2}}(Q_{r_j})} d_p((x_j,t_j),(0,0))^{2+\alpha_0} }{d_p((x_j,t_j),(0,0))^{\beta}}
\nonumber\\
&\leq \phi_j(d_p((x_j,t_j),(0,0))/2)
\leq \lim_{\rho \to 0} \phi_j(\rho).
\end{align}
Hence, \(
\displaystyle \lim_{j\to\infty}\lim_{\rho \to 0} \phi_j(\rho) = + \infty.
\) Moreover, by the definition of $\phi_j$, for $j \geq 3$, there exists $r_j \in (1/j,1/2]$ such that
\begin{equation}\label{eq:wj}
r_j^{2+ \alpha_0-\beta} [u_j]_{\mathcal{C}^{2+\alpha_0,\frac{2+\alpha_0}{2}}(Q_{r_j})}
\geq \frac{1}{2}\phi_j(1/j)
\geq \frac{1}{2}\phi_j(r_j).
\end{equation}
Therefore, from \eqref{eq:bound} we deduce
\[
r_j^{\beta-(2+\alpha_0)}
\leq \frac{2[u_j]_{\mathcal{C}^{2+\alpha_0,\frac{2+\alpha_0}{2}}(Q_{r_j})}}{\phi_j(1/j)}
\leq \frac{2\mathrm{K}_0}{\phi_j(1/j)}
\to 0 \quad \text{as } j \to \infty.
\]
As $2+\alpha_0 < \beta$, it follows that $r_j \to 0$.

In addition, defining the blow-up sequence
\[
w_j(x,t) := \frac{u_j(r_j x,r_j^2t)}{r_j^{\beta}\phi_j(r_j)},
\qquad x \in Q_{1/2r_j},
\]
we obtain, 
\begin{equation}\label{eq:w_j(0)}
w_j(0,0) = |Dw_j(0,0)| = |D^2w_j(0,0)|= |\partial_{t}w_j(0,0)| = 0,
\end{equation}
and in view of \eqref{eq:wj}
\begin{equation}\label{eq:wjnontrivial}
[w_j]_{C^{2+\alpha_0, \frac{2+\alpha_0}{2}}(Q_1)} = \frac{r_j^{2+\alpha_0-\beta}}{\phi_j(r_j)} [u_j]_{C^{2+\alpha_0, \frac{2+\alpha_0}{2}}(Q_{r_j})} \ge \frac{1}{2}.
\end{equation}
Fixing $r>1$, there exists $j_{0}\in\mathbb{N}$ such that $Q_{r}\subsetneq Q_{\frac{1}{2r_{j}}}$ for all $j\geq j_{0}$. In this case, arguing as in \eqref{estwithuj} and using that $(0,0)\in \Gamma_2(u_{j})$, we have for any $(x,t)\in Q_{r}$ that
\begin{equation}\label{eq:w_j}
|w_{j}(x,t)|\leq \frac{[u_{j}]_{\mathcal{C}^{2+\alpha_{0},\frac{2+\alpha_{0}}{2}}(Q_{r_jr})}(r_{j}r)^{2+\alpha_{0}-\beta}}{\phi_{j}(r_{j})}r^{\beta}\leq r^{\beta},
\end{equation}
since $r_{j}r\in (r_{j},1/2)$. Hence, $
\|w_{j}\|_{L^{\infty}(Q_{r})}\leq r^{\beta}.
$
Similarly, it is possible to check that
\begin{align*}
&\|Dw_{j}\|_{L^{\infty}(Q_{r})}\leq r^{\beta-1},\quad \|D^2w_{j}\|_{L^{\infty}(Q_{r})}\leq r^{\beta-2},\quad \\
&\|\partial_{t} w_{j}\|_{L^{\infty}(Q_{r})}\leq r^{\beta-2}
\quad\text{and}\quad[w_{j}]_{\mathcal{C}^{2+\alpha_{0},\frac{2+\alpha_{0}}{2}}(Q_{r})}\leq r^{\beta-(2+\alpha_{0})}    
\end{align*}
Thus, $\|w_{j}\|_{\mathcal{C}^{2+\alpha_{0},\frac{2+\alpha_{0}}{2}}(Q_{r})}\leq 5r^{\beta}$, since $r>1$. In addition, it $w_j$ solves, in the viscosity sense,
\[
F_j(D^2 w_j, x,t)-\partial_{t}w_j
+ \langle  B_j(x,t), Dw_j \rangle
=  \mathscr{F}_j(x,t,w^+_j,w^-_j)  \ \mbox{ in }  \ Q_{1/2r_j},
\]
 where $\mathscr{F}_j(x,t,w_j): =  f_j(r_jx,r_j^2t)r_j^{2 - \beta(1 - \gamma)} \phi_j(r_j)^{\gamma-1}\Big( (w_j^+)^\gamma - (w_j^-)^\gamma \Big),$ and
\[
F_j(M,x,t) := \frac{r_j^{2 - \beta}}{\phi_j(r_j)} F \left( \phi_j(r_j) r_j^{\beta - 2} M, r_j x, r_j^2 t \right) \mbox{ and } B_j(x,t) := r_jB(r_jx,r_j^2t).
\]
Since $r_j \to 0$, it follows that
\[
| \langle B_j( x,t), Dw_j \rangle| \leq r_j\|B\|_{L^{\infty}(Q_{1};\mathbb{R}^{n})}\|Dw_j\|_{L^\infty(Q_1)}\to 0, \mbox{ as } j \to \infty.
\]
Moreover, since $\gamma \in (0,1)$ and $\displaystyle \lim_{\rho \to 0} \phi_j(\rho) = + \infty,$ using condition ($\rm A_{2}$) combined with $(0,0)\in f_{j}^{-1}(0)$ we have that
\begin{equation}\label{eq:g_j}
f_{j}(r_{j}x,r_j^2t)\leq \mathrm{C}_{0}r_{j}^{\theta}d_p(x,t),(0,0))^{\theta}.
\end{equation}
Then, using \eqref{eq:wj} and \eqref{eq:g_j}, and since $0<\gamma <1$, as  $j \to \infty,$ we obtain
\[|\mathscr{F}_j(x,t,w_j)| \leq \mathrm{C}_{0}r_{j}^{\theta}d_p(x,t),(0,0))^{\theta}
2\phi_j(r_j)^{\gamma-1}  |w_j(x,t)|^\gamma \leq 2\mathrm{C}_{0}\phi_j(r_j)^{\gamma-1}r_j^{\beta\gamma+\theta} \to 0.
\]
   Thus, $(w_j) \subset \mathcal{G}^{\alpha_0}_2$ is bounded, and  there exists $w_0 \in \mathcal{C}^0(Q_{1/2})$ such that, up to subsequences, $ w_j \to w_0$, locally uniformly, and by the stability result, Lemma \ref{Est},  $w_0$ solves
\[
\mathrm{tr}(\mathfrak{A}_0(0,0)D^2 w_0)-\partial_{t}w_0 
= 0
\quad \text{in } \mathbb{R}^n \times (-\infty,0],
\]
where $
(\mathfrak{A}_0)_{ij}(0,0) := \frac{\partial F}{\partial X_{ij}}(\mathrm{O}_n,0,0),
\ \text{for } 1 \leq i,j \leq n, 
$
 is constant and uniformly elliptic. This is a linear uniformly parabolic equation with constant coefficients. Thus, $w_0$ is a smooth function (see \cite[\S 2.3, Theorem 8]{Evans10}. Furthermore, by the interior estimates for the heat equation \cite[\S 2.3, Theorem 9]{Evans10},  for any parabolic cylinder
$Q_r := B_r \times (-r^2,0],$
we have
\[
\sum_{2\rho+|\nu| = m}\|D^\rho_tD^\nu_x w_0\|_{L^\infty(Q_{r/2})}\le\frac{C}{r^m} \|w_0\|_{L^\infty(Q_r)}.
\]
Since
the limit equation holds in $\mathbb{R}^n \times (-\infty,0]$, we may take $r \to + \infty$. Thus, taking into account that $
\|w_0\|_{L^\infty(Q_r)}\leq r^\beta \leq r^3
$, for $m\geq 3$, we deduce
\[
\|D^\rho_tD^\nu_x w_0\|_{L^\infty(Q_{r/2})} \to 0, \quad \mbox{as} \ r \to +\infty.
\]
Hence,
\[
D^\rho_tD^\nu_x w_0 = 0,\quad \text{for} \quad 2\rho+|\nu|\geq 3.
\]
Therefore, using  the parabolic Taylor polynomial  expansion of $w_0$, we conclude that $w_0$ is locally a quadratic polynomial function  in  space and time. In addition, in view of \eqref{eq:w_j(0)}, we conclude that $w_0(0,0) = |Dw_0(0,0)| = |D^2w_0(0,0)|= |\partial_{t}w_0(0,0)| = 0,$
which means that
$
w_0 = 0,
$
in a neighborhood of $(0,0)$. By connectedness we conclude that $w_0\equiv 0$. This contradicts the inequality \eqref{eq:wjnontrivial}.
Therefore, the desired estimate holds.
\end{proof}

We now investigate the regularity of solutions to \eqref{P} when the source term $f$ depends only on the variables $x$ and $t$. We focus on when the source term satisfies a non-degeneracy condition as presented in \eqref{coninfdhol}. More precisely, we establish Theorem \ref{Thm1.2}.

\begin{proof}[\bf Proof of Theorem \ref{Thm1.2}]
Let $(x_{0},t_{0})\in Q_{1}$ be  a local extremum point. Without loss of generality, we may assume that $(x_{0},t_{0})=(0,0)$. 
We now introduce the following comparison (barrier) function
\begin{equation*}
v(x,t)=\mathrm{C}\left[|x|^{1+\sigma}+(-t)^{\frac{1+\sigma}{2}}\right],
\end{equation*}
where $\sigma\in(1,2]$ and $\mathrm{C}>0$ will be chosen \textit{a posteriori} so that $v$ is a viscosity solution to
\begin{equation}\label{est1Theorem1.2}
\mathscr{G}[v]:=\partial_{t}v-F(D^{2}v,x,t)+\langle B(x,t),Dv\rangle\geq -\mathrm{c}_f d_p((x,t),(0,0))^\theta \quad \text{in}\quad Q_{r}
\end{equation}
for all $r>0$ such that $Q_{r}\subset\subset Q_{1}$. To this end, note that $v$ satisfies
\begin{itemize}
\item[] $\partial_{t}v=-\mathrm{C}\left(\frac{1+\sigma}{2}\right)(-t)^{\frac{\sigma-1}{2}}$,
\quad \(D v = \mathrm{C}(1+\sigma)|x|^{\sigma-1}x\), { and}
\item[] \(D^{2} v = \mathrm{C}(1+\sigma)|x|^{\sigma-1}\left(\mathrm{Id}_n - (1-\sigma)\frac{x}{|x|} \otimes \frac{x}{|x|}\right)\),
\quad \(\mathcal{M}^{+}_{\lambda,\Lambda}(D^{2}v)=\Lambda\mathrm{C}(1+\sigma)(n-1+\sigma)|x|^{\sigma-1}\).
\end{itemize}
Under this setting, by the uniform ellipticity {$(\mathrm{A}_1)$}, we can conclude that
\begin{eqnarray}
F(D^{2}v,x,t)-\langle B(x,t),Dv\rangle&\leq& \mathcal{M}^{+}_{\lambda}(D^{2}v)+\|B\|_{L^{\infty}(Q_{1};\mathbb{R}^{n})}|Dv| \nonumber\\
&\leq&\mathrm{C}(1+\sigma)[\Lambda(n-1+\sigma)+\|B\|_{L^{\infty}(Q_{1};\mathbb{R}^{n})}]|x|^{\sigma-1}. \label{est2Theorem1.2}
\end{eqnarray}
Now, chosing $
\sigma=1+\theta, $
 and using the expression for $\partial_tv$ together with \eqref{est2Theorem1.2}, we conclude that
\begin{equation}\label{est3Theorem1.2}
\mathscr{G}[v]\geq -\mathrm{C}(2+\theta)\left(\frac{1}{2}+\Lambda(n+\theta)+\|B\|_{L^{\infty}(Q_{1};\mathbb{R}^{n})}\right)d_{p}((x,t),(0,0))^{\theta}.
\end{equation}
Thus, by choosing $\mathrm{C}$ such that
\begin{equation*}
0<\mathrm{C}\leq \frac{\mathrm{c}_{f}}{(2+\theta)\left(\frac{1}{2}+\Lambda(n+\theta)+\|B\|_{L^{\infty}(Q_{1};\mathbb{R}^{n})}\right)},
\end{equation*}
in view of \eqref{est3Theorem1.2} and assumption \eqref{coninfdhol}, we obtain
\begin{equation*}
\mathscr{G}[v]\geq -\mathrm{c}_{f}d_{p}((x,t),(0,0))^{\theta}\geq f(x,t). 
\end{equation*}
Therefore, $v$ is a supersolution to \eqref{problemwithoutdepofu}. 

On the other hand, for $\varepsilon>0$, we define the auxiliary function
\[
u_{\varepsilon}(x,t):=u(x,t)-u(0,0)+\varepsilon.
\]
Clearly, $u_{\varepsilon}$ is a viscosity solution to \eqref{problemwithoutdepofu} and $u_{\varepsilon}(0,0)=\varepsilon$. We claim that there exists a point $(x_{r},t_{r})\in \partial_{p}Q_{r}$ such that
\[
u_{\varepsilon}(x_{r},t_{r})\geq v(x_{r},t_{r}).
\]
Indeed, suppose by contradiction that $u_{\varepsilon}<v$ on $\partial_{p}Q_{r}$. In this case, since $v$ is a viscosity supersolution in $Q_{r}$ and $u_{\varepsilon}$ is a viscosity subsolution in $Q_{r}$, the Comparison Principle,   Theorem \ref{CompPrinc}, yields
\[
u_{\varepsilon}(x,t)\leq v(x,t)\quad \text{for any}\quad(x,t)\in Q_{r}.
\]
In particular, for $(x,t)=(0,0)$ we have 
$
\varepsilon=u_{\varepsilon}(0,0)\leq v(0,0)=0,
$
which is a contradiction. This proves the claim. Thus,
\begin{equation}\label{est4Theorem1.2}
\sup_{\partial_{p}Q_{r}}u_{\varepsilon}\geq v(x_{r},t_{r})=\mathrm{C}[|x_{r}|^{2+\theta}+(-t_{r})^{\frac{2+\theta}{2}}].
\end{equation}
Now, by the definition of the parabolic boundary, there are two cases to analyze for $(x_{r},t_{r})$:
\begin{itemize}
\item[I.]  In this case, $(x_{r},t_{r})\in \partial B_{r}\times[-r^{2},0)$, then $|x_{r}|=r$ and since $(-t_{r})^{\frac{2+\theta}{2}}\geq 0$, it follows that
\[
v(x_{r},t_{r})
=\mathrm{C}\Big(r^{2+\theta}+(-t_{r})^{\frac{2+\theta}{2}}\Big)
\geq \mathrm{C}r^{2+\theta}.
\]

\item[II.] In this case, $(x_{r},t_{r})\in B_{r}\times\{-r^{2}\}$, then $t_{r}=-r^{2}$ and since $|x_{r}|^{2+\theta}\geq 0$, thus,
\[
v(x_{r},t_{r})
=\mathrm{C}\Big(|x_{r}|^{2+\theta}+(-(-r^{2}))^{\frac{2+\theta}{2}}\Big)
\geq \mathrm{C}r^{2+\theta}.
\]

\end{itemize}
In view of these two cases and \eqref{est4Theorem1.2}, we conclude that
\begin{equation*}
\sup_{\partial_{p}Q_{r}}|u(x,t)-u(0,0)|\geq\sup_{\partial_{p}Q_{r}}(u(x,t)-u(0,0))=\sup_{\partial_{p}Q_{r}}u_{\varepsilon}-\varepsilon\geq \mathrm{C}r^{2+\theta}-\varepsilon,
\end{equation*}
letting $\varepsilon\to 0$ we conclude the desired.
\end{proof}

To conclude this manuscript, we present a proof of Theorem \ref{Thm1.4}.

\begin{proof}[\bf Proof of Theorem \ref{Thm1.4}]
We treat only the case in which the decay of the positive part of $u$ is assumed; the modifications required for the complementary case are briefly indicated at the end. Without loss of generality, we assume $(x_{0},t_{0})=(0,0)$. Following \cite{CKS00}, we first establish the existence of a constant $\bar{\mathrm{C}}>0$ such that
\begin{equation}\label{4.7}
s_{j+1}\leq \max\{{\bar{ \mathrm{C}}}2^{-{\alpha}(j+1)},2^{-{\alpha}}s_{j}\},
\end{equation}
where
\begin{equation*}
s_{j}=\|u\|_{L^{\infty}(Q_{2^{-j}})}\qquad \text{and}\qquad {\alpha}=\frac{2+\theta}{1-\gamma}.
\end{equation*}
To prove \eqref{4.7}, we argue by contradiction. For each $k\in\mathbb{N}$, suppose there exists $j_{k}\in\mathbb{N}$ such that
\begin{equation}\label{4.8}
s_{j_{k}+1}> \max\{k2^{-{\alpha}(j_{k}+1)},2^{-{\alpha}}s_{j_{k}}\}.
\end{equation}
In this setting, define the rescaled profile $u_{k}:Q_{1}\to\mathbb{R}$ by
\begin{equation*}
u_k(x,t) \coloneqq \frac{u(2^{-j_k} x,2^{-2j_{k}}t)}{s_{j_k + 1}}.
\end{equation*}
For $u_{k}$, we have $u_{k}(0,0)=0$, since $(0,0)\in \Gamma_2(u)$, $\|u_{k}\|_{L^{\infty}(Q_{1/2})}=1$, and by \eqref{4.8}, $\|u_{k}\|_{L^{\infty}(Q_{1})}\leq 2^{{\alpha}}$. Moreover, $u_{k}$ satisfies, in the viscosity sense,
\begin{equation*}
F_{k}(D^{2}u_{k},x,t)-\partial_{t}u_{k}+\langle B_{k}(x,t),Du_{k}\rangle=\mathscr{F}_{k}(x,t,u^+_{k}, u^{-}_{k})
\qquad \text{in} \qquad Q_{1},
\end{equation*}
where
\begin{eqnarray*}
\left\{
\begin{array}{lll}
F_{k}(\mathrm{M},x,t) &:=& \frac{2^{-2j_{k}}}{s_{j_{k}+1}} F\left(\frac{s_{j_{k}+1}}{2^{-2j_{k}}}\mathrm{M},2^{-j_{k}}x,2^{-2j_{k}}t\right), \\
B_{k}(x,t) & := & 2^{-j_{k}}B(2^{-j_{k}}x,2^{-2j_{k}}t), \\
\mathscr{F}_{k}(x,t,u^{+}_{k}, u^{-}_{k} )&:=& \frac{2^{-2j_{k}}}{s_{j_{k}+1}^{1-\gamma}}f(2^{-j_{k}}x,2^{-2j_{k}}t)((u_{k}^{+})^{\gamma}-(u_{k}^{-})^{\gamma}).
\end{array}
\right.
\end{eqnarray*}
Note that $B_{k}\to 0$ as $k\to\infty$, since $B$ is bounded. Moreover, using that $(0,0)\in f^{-1}(0)$, combining assumption $\rm (A_2)$ with \eqref{4.8},  we obtain
\begin{equation*}
|\mathscr{F}_{k}(x,t,u^{+}_{k}, u^{-}_{k})|\leq \frac{2^{-2j_{k}}}{s_{j_{k}+1}^{1-\gamma}}2^{1+\gamma{\alpha}}\mathrm{C}_02^{-j_{k}\theta}\leq\mathrm{C}_0\frac{2^{3+\gamma{\alpha}+\theta}}{k^{1-\gamma}} \to 0 \quad \mbox{as} \ k \to \infty.
\end{equation*}
Finally, it is straightforward to verify that $F_{k}$ satisfies the same structural assumptions in {($\mathrm{A}_1$)} and that, 
for every $\mathrm{X}\in \text{Sym}(n)$, 

\begin{eqnarray*}
\frac{|F_{k}(\mathrm{X},x,t)-F_{k}(\mathrm{X},0,0)|}{1+\|\mathrm{X}\|}&=&\frac{2^{-2j_{k}}}{s_{j_{k}+1}}\frac{|F\left(\frac{s_{j_{k}+1}}{2^{-2j_{k}}}\mathrm{X},2^{-j_{k}}x,2^{-2j_{k}}t\right)-F\left(\frac{s_{j_{k}+1}}{2^{-2j_{k}}}\mathrm{X},0,0\right)|}{1+\|\mathrm{X}\|} \\
&\leq&\frac{2^{-2j_{k}}}{s_{j_{k}+1}}\frac{\left(1+\frac{s_{j_{k}+1}}{2^{-2j_{k}}}\|\mathrm{X}\|\right)\Phi_{F}(2^{-j_{k}}x,2^{-2j_{k}}t)}{1+\|\mathrm{X}\|}\\
&=&\frac{2^{-2j_{k}}}{s_{j_{k}+1}}\frac{\frac{s_{j_{k}+1}}{2^{-2j_{k}}}\left(\frac{2^{-2j_{k}}}{s_{j_{k}+1}}+\|\mathrm{X}\|\right)}{1+\|\mathrm{X}\|}\Phi_{F}(2^{-j_{k}}x,2^{-2j_{k}}t)\\
&=&\frac{\left(\frac{2^{-2j_{k}}}{s_{j_{k}+1}}+\|\mathrm{X}\|\right)}{1+\|\mathrm{X}\|}\Phi_{F}(2^{-j_{k}}x,2^{-2j_{k}}t)\\
&\leq&\left(\frac{2^{-2j_{k}}}{s_{j_{k}+1}}+1\right)\Phi_{F}(2^{-j_{k}}x,2^{-2j_{k}}t).
\end{eqnarray*}
Taking the supremum over $\mathrm{X} \in \text{Sym}(n)$, it follows that
\[
\Phi_{F_{k}}(x,t)\leq \left(1+\frac{2^{-2j_{k}}}{s_{j_{k}+1}}\right)\Phi_{F}(2^{-j_{k}}x,2^{-2j_{k}}t).
\]
By \eqref{4.8} and the structural condition $\rm (A_3)$, it yields
\begin{equation}\label{4.9}
\Phi_{F_{k}}(x,t)\leq \left(1+\frac{2^{\alpha}2^{j_{k}(\alpha-2)}}{k}\right)\Phi_{F}(2^{-j_{k}}x,2^{-2j_{k}}t)\leq\mathrm{c}_{\Phi} \left(\frac{1}{2^{j_{k}\zeta}}+\frac{2^{\alpha}}{k}\right)
\end{equation}
which converges to zero as $k\to \infty$. Therefore, we may apply the stability result (Lemma \ref{Est}) and, up to subsequences, obtain $u_{k}\to u_{\infty}$ and $F_{k}\to F_{\infty}$, where $u_{\infty}$ is a viscosity solution to
\begin{equation}\label{4.10}
F_{\infty}(D^{2}u_{\infty})-\partial_{t}u_{\infty}=0\qquad\text{in}\qquad Q_{3/4},
\end{equation}
and $F_{\infty}$ is an operator satisfying {\rm ($\mathrm{A}_1$)} with constant coefficients, in light of \eqref{4.9}. From the properties established above, it follows that $u_{\infty}$ satisfies
\begin{equation}\label{4.11}
u_{\infty}(0,0)=0\qquad \text{and}\qquad
\|u_{\infty}\|_{L^{\infty}(Q_{1/2})}=1.
\end{equation}

Now, for $k\gg 1$ such that $Q_{2^{-j_{k}}}\subset Q_{r_{0}}$, the decay assumption on $u^{+}$ yields
\begin{equation*}
0\leq u_k^{+}(x,t) \leq \frac{\mathrm{C}_{0} 2^{-j_k {\alpha}}}{s_{j_k + 1}}
\leq  \frac{2^{{\alpha}} \mathrm{C}_0}{k},
\end{equation*}
where the last inequality follows from \eqref{4.8}. Passing to the limit, we conclude that $u_{\infty}^{+}=0$, hence $u_{\infty}\leq 0$. Therefore, $(0,0)$ is a nonnegative maximum point of $u_{\infty}$, which solves \eqref{4.10}. Invoking the Strong Maximum Principle (see \cite[Theorem 2.1]{DaL04}), we deduce that $u_{\infty}$ is constant. Since $u_{\infty}(0,0)=0$, it follows that $u_{\infty}\equiv 0$, contradicting \eqref{4.11} and thereby proving \eqref{4.7}.

Finally, to obtain the desired result, let $j_0 \in \mathbb{N}$ be the smallest integer such that $2^{-j_0} \leq r_0$. Then
\begin{equation}\label{4.12}
s_{j_0} \leq \mathrm{C} \, 2^{-{\alpha} j_0},
\end{equation}
where $\mathrm{C} = \max\{\bar{\mathrm{C}},\, 2^{{\alpha}j_0}s_{j_0}\}$. For any $r \in (0,r_0]$, choose $j\geq j_0$ such that
\[
2^{-(j+1)}\leq r\leq 2^{-j}.
\]
Using \eqref{4.12}, we obtain
\begin{equation}\label{4.13}
\sup_{Q_r} u^{-}
\leq \|u\|_{L^{\infty}(Q_r)}+\mathrm{C}_1 r^{{\alpha}}
\leq s_j+\mathrm{C}_1 r^{{\alpha}}
\leq \mathrm{C}\,2^{-{\alpha}j}+\mathrm{C}_1 r^{{\alpha}}
\leq \left(2^{{\alpha}}{\mathrm{C}}+\mathrm{C}_1\right)r^{{\alpha}}.
\end{equation}
Setting $\mathrm{C}_2:=2^{{\alpha}}\mathrm{C}+\mathrm{C}_1$, we conclude that
\[
\sup_{Q_r}u^{-}\leq \mathrm{C}_2 r^{{\alpha}} \quad \text{for every } r\in(0,r_0].
\]

The remaining case is handled by the same contradiction argument. Specifically, assuming the growth condition on $u^{-}$, one constructs a blow-up profile $u_{\infty}$ whose negative part vanishes identically, while $(0,0)$ becomes a non-positive minimum point of a solution to the limit equation \eqref{4.10}. This leads to a contradiction via the Strong Minimum Principle, established in Da Lio \cite{DaL04}. The decay estimate for the positive part then follows by repeating the argument in \eqref{4.13} with the necessary modifications. This completes the proof.

\end{proof}

\section{Final remarks and examples}

In this final section, we discuss several conclusions and present explicit examples of solutions attaining the sharp regularity predicted by our main results.

\begin{remark}[{\bf Improved estimates at higher-order singular nodal points}]
The blow-up argument developed in Theorem \ref{Thm1.3} extends to higher-order singular nodal points. Indeed, if $(x_{0},t_{0})\in\Gamma_k(u)$, the vanishing of all parabolic derivatives of $u$ at $(x_{0},t_{0})$ up to order $k$ allows the quadratic Taylor expansion used for $k=2$ to be replaced by the parabolic Taylor expansion of order $k\geq3$. We then introduce
$$
\phi_j(\rho)
:=
\sup_{\rho<r<1/2}
r^{k+\alpha_0-\beta}
[u_j]_{\mathcal{C}^{k+\alpha_0,\frac{k+\alpha_0}{2}}(Q_r)}
$$
and the rescaling
$$
w_j(x,t)
:=
\frac{u_j(r_jx,r_j^2t)}
{r_j^\beta\phi_j(r_j)}.
$$
The assumptions
$$
\beta>k+\alpha_0
\qquad\text{and}\qquad
\theta>(k-2) +\alpha_0 -(k+\alpha_0)\gamma
$$
ensure, respectively, that $r_j\to0$ and that the admissible range for $\beta$ is nonempty. Moreover, the condition
$$
\beta\leq\frac{2+\theta}{1-\gamma}
$$
ensures that the singular right-hand side vanishes in the rescaled equation, while the lower-order drift term becomes negligible, since $B_j(x,t)=r_jB(r_jx,r_j^2t)\to0$ locally uniformly.

The limiting function $w_0$ therefore solves a linear uniformly parabolic equation with constant coefficients in $\mathbb{R}^n\times(-\infty,0]$ and satisfies the global growth estimate
$$
\|w_0\|_{L^\infty(Q_R)}\leq \mathrm{C} R^\beta.
$$
In addition, the condition
$
\beta\leq\frac{2+\theta}{1-\gamma}<k+1,
$
combined with the interior estimates for the limiting equation implies that all parabolic derivatives of $w_0$ of order greater than $k$ vanish identically. Hence,  $w_0$ is locally a parabolic polynomial of degree at most $k$. Since $(x_0,t_0)\in\Gamma_k(u_j)$, all parabolic derivatives of $w_0$ of orders up to $k$ vanish at $(x_0,t_0)$, and therefore
$
w_0\equiv0.
$
This contradicts the non-triviality inherited from the normalization of the blow-up sequence.
Thus, the above argument yields the improved estimate for
$$
k+\alpha_0<\beta<
\frac{2+\theta}{1-\gamma}.
$$
The structural restriction $
\frac{2+\theta}{1-\gamma}< k+1
$ is equivalent to
$
\theta<(k-1)-(k+1)\gamma,
$  which generalizes the condition $\theta < 1 - 3\gamma$  imposed when $k=2.$
\end{remark}

\bigskip

In what follows, we present two examples illustrating the optimality of the regularity exponent.

\begin{example}[{\bf Sharpness of the critical exponent in the two-phase setting}]
Let $0<\gamma<1$, $\theta\geq0$, and define
$
    \alpha:=\frac{2+\theta}{1-\gamma}$, in particular, $\alpha>2$. We consider the two-phase equation
$$   \Delta u
    =
    f(x)
    \big[(u^+)^\gamma-(u^-)^\gamma\big]
    \qquad\text{in }B_1\subset\mathbb R^2.$$
We define
$
    u(x):=\operatorname{sgn}(x_1)|x_1|^\alpha.
$
Then, 
$$
    u^+(x)=
    \begin{cases}
        |x_1|^\alpha,&x_1>0,\\
        0,&x_1\leq0,
    \end{cases}
    \qquad
    u^-(x)=
    \begin{cases}
        0,&x_1\geq0,\\
        |x_1|^\alpha,&x_1<0.
    \end{cases}
$$
Consequently,
$
    (u^+)^\gamma-(u^-)^\gamma
    =
    \operatorname{sgn}(x_1)|x_1|^{\alpha\gamma}.
$
Since $u$ depends only on $x_1$,
$$
    D^2u
    =
    \begin{pmatrix}
        \alpha(\alpha-1)|x_1|^{\alpha-2}&0\\
        0&0
    \end{pmatrix}
    \operatorname{sgn}(x_1),
\quad \mbox{ and hence } \quad 
    \Delta u
    =
    \alpha(\alpha-1)
    \operatorname{sgn}(x_1)|x_1|^{\alpha-2}.
$$
Now define
$
    f(x)
    :=
     \alpha(\alpha-1)|x_1|^\theta.
$
Then,
$$
\begin{aligned}
    f(x)
    \big[(u^+)^\gamma-(u^-)^\gamma\big]
    &=
   \alpha(\alpha-1)|x_1|^\theta
    \operatorname{sgn}(x_1)|x_1|^{\alpha\gamma}
=
   \alpha(\alpha-1)
    \operatorname{sgn}(x_1)
    |x_1|^{\theta+\alpha\gamma}.
\end{aligned}
$$
By the definition of $\alpha$, it follows that
$  \alpha(1-\gamma)=2+\theta,
$
and, therefore,
$
   \alpha-2=\theta+\alpha\gamma.$
This yields,
$$
{
    \Delta u
    =
    f(x)
    \big[(u^+)^\gamma-(u^-)^\gamma\big]
    \qquad\text{in }B_1.
    }
$$
Notice that $u$ is genuinely two-phase, 
since
$
    \{u>0\}=\{x_1>0\}, \
    \{u<0\}=\{x_1<0\},
$
and the two phases meet along
$
    \Gamma(u)=\{x_1=0\}.
$
Moreover,
$
    u(0)=0,
    \
    Du(0)=0,
    \
    D^2u(0)=0,
$
whereas the $\alpha$-th order behavior is nontrivial. More precisely, $
    |u(x)|=|x_1|^\alpha$. Hence, for every $r>0$, it follows that
$$
    \sup_{B_r}|u|
    =
    r^\alpha.
$$
Additionally, for every $\widetilde\alpha>\alpha$, one has

$$
    \frac{\sup_{B_r}|u|}{r^{\widetilde\alpha}}
    =
    r^{\alpha-\widetilde\alpha}
    \to +\infty
    \qquad\text{as }r\to0^+.
$$
Therefore, no pointwise growth estimate of order strictly larger than $
    \alpha=\frac{2+\theta}{1-\gamma}
$ can hold at the two-phase point $x=0$ under the same structural assumptions.
This example thus shows that the exponent $
    \alpha = \frac{2+\theta}{1-\gamma}$ is sharp for the two-phase problem.
\end{example}

\medskip

\begin{example}[{\bf A sharp two-phase time-dependent self-similar profile}]
Let $\mathfrak A_0$ be a symmetric, uniformly elliptic matrix, let
$\mathbf e\in\mathbb R^n\setminus\{0\}$, and set $
    \lambda_{\mathfrak A_0}
    :=\mathbf e^T\mathfrak A_0\mathbf e>0.
$
We consider the two-phase parabolic equation
$$
    \operatorname{Tr}(\mathfrak A_0D^2w)-\partial_t w
    =(w^+)^\gamma-(w^-)^\gamma
    \quad\text{in }\mathbb R^n\times(-\infty,0),
$$
where $0<\gamma<1$.
We seek a parabolically self-similar solution of the form

$$
    w(x,t)
    =
    (-t)^{p}
    \Psi\left(
        \frac{\mathbf e\cdot x}{\sqrt{-t}}
    \right),
    \qquad t<0,
$$
where $p>0$ and $\Psi:\mathbb R\to\mathbb R$ are to be determined. Set $
    z:=\frac{\mathbf e\cdot x}{\sqrt{-t}},
$ then $
    Dz=\frac{\mathbf e}{\sqrt{-t}}$, 
$$   Dw
    =
    (-t)^{p-\frac12}\Psi'(z)\mathbf e, \quad  D^2w
    =
    (-t)^{p-1}
    \Psi''(z)\mathbf e\otimes\mathbf e, \mbox{ and hence, }
    \operatorname{Tr}(\mathfrak A_0D^2w)
    =
    \lambda_{\mathfrak A_0}
    (-t)^{p-1}\Psi''(z).
$$
Moreover, \(
    \partial_t(-t)^{p}
    =
    -p(-t)^{p-1}\),
 \(    \partial_tz
    =
    \frac{z}{2(-t)}\),
and consequently
$
    \partial_t w
    =
    (-t)^{p-1}
    \left(
        -p\Psi(z)
        +\frac z2\Psi'(z)
    \right).
$
Therefore, it follows that

$$
\begin{aligned}
\operatorname{Tr}(\mathfrak A_0D^2w)-\partial_t w
&=
(-t)^{p-1}
\left[
    \lambda_{\mathfrak A_0}\Psi''(z)
    +p\Psi(z)
    -\frac z2\Psi'(z)
\right].
\end{aligned}
$$

On the other hand,
$
    (w^+)^\gamma-(w^-)^\gamma
    =
    (-t)^{p\gamma}
    \left[
        (\Psi^+)^\gamma-(\Psi^-)^\gamma
    \right].
$
Equating the homogeneities, we obtain \(
    p-1=p\gamma\),
    and hence \(
    p=\frac{1}{1-\gamma}.
\)
Thus, the corresponding spatial homogeneity is $
    \alpha=2p=\frac{2}{1-\gamma}.
$
The similarity profile therefore satisfies the nonlinear ODE
\begin{equation}\label{eq:similarity-two-phase}
\lambda_{\mathfrak A_0}\Psi''(z)
+\frac{1}{1-\gamma}\Psi(z)
-\frac z2\Psi'(z)
=
(\Psi^+)^\gamma-(\Psi^-)^\gamma.
\end{equation}

We now seek a genuinely two-phase profile with a moving interface.
Fix $z_\ast\in\mathbb R\setminus{0}$ and consider a solution of
\eqref{eq:similarity-two-phase} satisfying
$
    \Psi(z_\ast)=\Psi'(z_\ast)=0$,
{with}  $
    \Psi(z)<0 $ for  $z<z_\ast, $ and $
    \Psi(z)>0 $ {for } $ z>z_\ast.
$
Thus, $z=z_\ast$ represents a singular nodal point of the two phases in the
similarity variable.
To determine the leading-order behavior near $z_\ast$, we seek a profile such that
$$
    \Psi(z)
    \sim
    \begin{cases}
    -\mathrm C_-(z_\ast-z)^m,
        & z\uparrow z_\ast,\\[2mm]
    \mathrm C_+(z-z_\ast)^m,
        & z\downarrow z_\ast,
    \end{cases}
$$
where $\mathrm C_\pm>0$. Since
$$
    \Psi''(z)
    \sim
    \begin{cases}
    -\mathrm C_-m(m-1)(z_\ast-z)^{m-2},
        & z<z_\ast,\\[2mm]
    \mathrm C_+m(m-1)(z-z_\ast)^{m-2},
        & z>z_\ast,
    \end{cases}
\quad \mbox{ while } \quad
    (\Psi^+)^\gamma-(\Psi^-)^\gamma
    \sim
    \begin{cases}
    -\mathrm C_-^\gamma(z_\ast-z)^{m\gamma},
        & z<z_\ast,\\[2mm]
    \mathrm C_+^\gamma(z-z_\ast)^{m\gamma},
        & z>z_\ast,
    \end{cases}
$$
the dominant balance between the second-order term and the nonlinear
term gives
$
 m-2=m\gamma,$ consequently, $  m=\frac{2}{1-\gamma}=\alpha.
$
Comparing the leading coefficients yields
$$
    \lambda_{\mathfrak A_0}\mathrm C_\pm
   \frac{2(1+\gamma)}{(1-\gamma)^2}
    =
    \mathrm C_\pm^\gamma, \quad 
\text{and hence} \quad 
    \mathrm C_+=\mathrm C_-=\mathrm C_0
    =
    \left[
        \lambda_{\mathfrak A_0}\frac{2(1+\gamma)}{(1-\gamma)^2}
    \right]^{-\frac{1}{1-\gamma}}.
$$
Thus, provided the corresponding two-sided similarity profile exists,
\begin{equation}\label{eq:two-phase-profile-asymptotic}
\Psi(z)
=
\mathrm C_0
\operatorname{sgn}(z-z_\ast)
|z-z_\ast|^{\frac{2}{1-\gamma}}
+
o\left(
|z-z_\ast|^{\frac{2}{1-\gamma}}
\right) \quad \text{as} \quad  z\to z_\ast.
\end{equation}

Notice that \eqref{eq:two-phase-profile-asymptotic} is an asymptotic
expansion rather than an exact power-law representation of $\Psi$.
Indeed, when $z_\ast\neq0$, the drift term $
    -\frac z2\Psi'(z)
$ contains the contribution $
    -\frac{z_\ast}{2}\Psi'(z),
$ which is of lower order than the leading balance but prevents the pure
power profile from solving \eqref{eq:similarity-two-phase} exactly.
Substituting the similarity profile into the original variables gives
\begin{equation}\label{eq:two-phase-self-similar}
w(x,t)
=
(-t)^{\frac{1}{1-\gamma}}
\Psi\left(
\frac{\mathbf e\cdot x}{\sqrt{-t}}
\right).
\end{equation}
Then, the zero set of $w$ contains the moving interface
$
    \Gamma
    =
    \left\{
        (x,t)\in\mathbb R^n\times(-\infty,0):
        \mathbf e\cdot x=z_\ast\sqrt{-t}
    \right\}.
$
Moreover, \eqref{eq:two-phase-profile-asymptotic} implies the sharp
two-phase expansion
$$
\begin{aligned}
w(x,t)
&=
\mathrm C_0
\operatorname{sgn}
\left(
    \mathbf e\cdot x-z_\ast\sqrt{-t}
\right)
\left|
    \mathbf e\cdot x-z_\ast\sqrt{-t}
\right|^{\frac{2}{1-\gamma}} \quad
+
o\left(
    \left|
        \mathbf e\cdot x-z_\ast\sqrt{-t}
    \right|^{\frac{2}{1-\gamma}}
\right).
\end{aligned}
$$
In particular, the two phases meet at $\Gamma$, with the same optimal
order of vanishing.
Since $\alpha>2$ for every $0<\gamma<1$, at every point
$(x_0,t_0)\in\Gamma$ with $t_0<0$ one has
$$
    w(x_0,t_0)=0,
    \qquad
    Dw(x_0,t_0)=0,
    \qquad
    D^2w(x_0,t_0)=0,
    \qquad
    \partial_tw(x_0,t_0)=0.
$$
Thus, the moving interface consists of points at which the solution and
all spatial derivatives up to order two, as well as the time derivative,
vanish.

Finally, the self-similar structure gives the exact parabolic scaling
$
    w(rx,r^2t)
    =
    r^{\frac{2}{1-\gamma}}w(x,t).
$
Consequently, the natural growth rate near the two-phase singular nodal
interface is
$$
    |w(x,t)|
    \sim
    \operatorname{dist}_{\mathrm{p}}
    \bigl((x,t),\Gamma\bigr)^{\frac{2}{1-\gamma}},
$$
and the exponent $    \alpha=\frac{2}{1-\gamma}
$
is dictated by the intrinsic scaling of the equation.
This construction provides a genuinely two-phase, time-dependent
self-similar framework for exhibiting the sharpness of the regularity
exponent at singular nodal points. In contrast with the stationary example,
both the solution and the free boundary depend non-trivially on time.
\end{example}

\subsection*{Data Availability Statement} \small Data sharing is not applicable to this article as no datasets were generated or analyzed
	during the current study.

\subsection*{Conflict of Interest Declaration} \small The authors declare that they have no conflict of interest.

\subsection*{AI disclosure} \small
In preparing this paper, the authors used the free version of ChatGPT solely for reviewing, improving, and editing the manuscript. ChatGPT was not used in developing the structure of the paper, including the Introduction, or in formulating, developing, or validating the mathematical arguments and proofs. All mathematical content, arguments, proofs, and conclusions were conceived, developed, and independently verified by the authors.

\subsection*{Acknowledgments}

J. da Silva Bessa has been supported by FAPESP-Brazil under Grant No. 2023/18447-3.
J.V. da Silva has received partial support from CNPq-Brazil under Grant No. 307131/2022-0,  Chamada CNPq/MCTI Nº 10/2023 - Faixa B - Grupos Consolidados under Grant 420014/2023-3, and by  FAPESP-Brazil under the Grant No.  2025/09344-1-Special Programs-Special Projects-First Projects-Call for Proposals (2025)-1st Cycle. 
M. Soares has been supported by CNPq-Brazil under the Grants No. 303.154/2025-0 and  No. 152.786/2025-2.

\bigskip

	\noindent\textsc{Junior da Silva Bessa}\\
	Departamento de Química\\
	Faculdade de Educação de Itapipoca - FACEDI\\
	Universidade Estadual do Ceará - UECE\\
	Rua da Universidade, s/n°, CEP 62505-090, Itapipoca, CE, Brazil\\
	\noindent \texttt{silva.bessa@uece.br}
\bigskip

	\noindent\textsc{Jo\~{a}o Vitor da Silva}\\
	Departamento de Matem\'{a}tica\\
	Instituto de Matem\'{a}tica, Estat\'{i}stica e Computa\c{c}\~{a}o Cient\'{i}fica - IMECC\\
	Universidade Estadual de Campinas - Unicamp\\
	Rua S\'{e}rgio Buarque de Holanda, 651, CEP 13083-859, Campinas, SP, Brazil\\
	\noindent \texttt{jdasilva@unicamp.br}

	\bigskip
	
	\noindent\textsc{Mayra Soares}\\
	Departamento de Matem\'{a}tica,\\ Universidade de Bras\'{i}lia - UnB\\
	Instituto Central de Ci\^{e}ncias, Campus Darcy Ribeiro,\\
	70910-900, Asa Norte, Bras\'{i}lia, Distrito Federal, Brazil\\
	\noindent\texttt{mayra.soares@unb.br}
\end{document}